\documentclass[11pt,oneside,reqno]{amsart}
\usepackage[utf8]{inputenc}
\usepackage[pdftex]{graphicx}
\usepackage{amssymb}
\usepackage{amsmath}
\usepackage{amsfonts}
\usepackage{amsthm}
\usepackage{bm}
\usepackage{mathrsfs}
\usepackage{mathtools}
\usepackage{xcolor}
\usepackage{comment}
\usepackage{appendix}
\usepackage{enumerate}
\usepackage[left=2.5cm, right=2.5cm, bottom=3.2cm]{geometry}
\usepackage[stretch=30,shrink=30]{microtype}
\usepackage{stmaryrd}
\usepackage[normalem]{ulem} 
\usepackage{xcolor}
\definecolor{antiquefuchsia}{rgb}{0.57, 0.36, 0.51}
\definecolor{azure}{rgb}{0.0, 0.5, 1.0}

\usepackage[backref=page, colorlinks = true, linkcolor = black, citecolor = antiquefuchsia]{hyperref}

\renewcommand*{\backref}[1]{}
\renewcommand*{\backrefalt}[4]{%
    \ifcase #1 (Not cited.)%
    \or        (Cited on page~#2.)%
    \else      (Cited on pages~#2.)%
    \fi}

\makeatletter
\def\th@plain{%
	\thm@notefont{}
	\itshape 
}
\def\th@definition{%
	\thm@notefont{}
	\normalfont 
}
\makeatother

\numberwithin{equation}{section}

\newtheorem{theorem}{Theorem}[section]
\newtheorem{lemma}[theorem]{Lemma}
\newtheorem{proposition}[theorem]{Proposition}

\theoremstyle{definition}
\newtheorem{definition}[theorem]{Definition}

\newtheorem{remark}[theorem]{Remark}
\newtheorem{example}[theorem]{Example}

\newcommand{\n}{\mathbb{N}}
\newcommand{\s}{\mathbb{S}}

\newcommand{\z}{\mathbb{Z}}
\renewcommand{\r}{\mathbb{R}}
\renewcommand{\c}{\mathbb{C}}

\DeclareMathOperator{\Area}{Area}
\DeclareMathOperator{\divergence}{div}

\DeclareMathOperator{\Real}{Re}
\DeclareMathOperator{\Imag}{Im}
\DeclareMathOperator{\End}{End}
\DeclareMathOperator{\E}{E}

\usepackage[capitalise,nameinlink]{cleveref}
\crefformat{equation}{#2(#1)#3}
\crefrangeformat{equation}{#3(#1)#4 to~#5(#2)#6}
\crefmultiformat{equation}{#2(#1)#3}{ and~#2(#1)#3}{, #2(#1)#3}{ and~#2(#1)#3}

\newcommand{\PSL}{\operatorname{PSL}}
\newcommand{\PSO}{\operatorname{PSO}}

\newcommand{\D}{\mathbb{D}}

\title[Area rigidity for the regular representation of surface groups]{Area rigidity for the regular representation of surface groups}

\author[Riccardo Caniato]{Riccardo Caniato}
\address{California Institute of Technology, Department of Mathematics, 1200 E California Blvd, MC 253-37, CA 91125, United States of America}
\email{rcaniato@caltech.edu}
\author[Xingzhe Li]{Xingzhe Li}
\address{Cornell University, Department of Mathematics, 310 Malott Hall, Ithaca, NY 14853, United States of America}
\email{xl833@cornell.edu}
\author[Antoine Song]{Antoine Song}
\address{California Institute of Technology, Department of Mathematics, 1200 E California Blvd, MC 253-37, CA 91125, United States of America}
\email{aysong@caltech.edu}
\date{\today} 

\begin{document}

\begin{abstract} 
Let $\tilde{\Sigma}$ be the universal cover of a closed surface $\Sigma$ of genus at least $2$. We characterize equivariantly area-minimizing maps from $\tilde{\Sigma}$ to a Hilbert sphere, which are equivariant with respect to an isometric action of $\pi_1(\Sigma)$ weakly equivalent to the regular representation. This result implies that many models of random equivariant harmonic maps from surfaces to Euclidean spheres admit a unique geometric  limit with high probability as the dimension of the sphere goes to infinity.
As part of our proof, we classify complete totally isotropic minimal surfaces in Hilbert spheres with constant negative Gaussian curvature. This builds on earlier results of E. Calabi, K. Kenmotsu, R. Bryant.
\end{abstract}

\maketitle
\tableofcontents
\allowdisplaybreaks
%
%
\section{Introduction}\label{sec: introduction} 
Consider a closed oriented surface $\Sigma$ with fundamental group $\pi_1(\Sigma)$ and universal cover $\tilde{\Sigma}$. Given an isometric action $\rho$ of $\pi_1(\Sigma)$ on a Hilbert sphere $\s_H$, let $\mathscr{H}_{\Sigma,\rho}$ be the family of smooth $\pi_1(\Sigma)$-equivariant maps $\tilde{\Sigma} \to \s_H.$
\emph{For which $\rho$ does $\mathscr{H}_{\Sigma,\rho}$ contain an equivariantly area-minimizing map? For which $\rho$ are such maps unique up to reparametrization and ambient isometries?} We give a complete answer to these questions for a large family of isometric actions $\rho$, namely those which are ``weakly equivalent'' to the regular representation of $\pi_1(\Sigma)$. 

Our result is, to our knowledge, the first rigidity result of this kind for equivariant minimal surfaces in spheres. To give some perspective,  the rigidity of equivariant harmonic maps and minimal surfaces in nonpositively curved spaces is a well-studied subject \cite{Siu80, Corlette92, GS92, MSY93, JY93, KS93, KS97, Wang98, DMV11, BDM24}, and the basis of fundamental developments in geometry and topology.
On the other hand, when the ambient space is positively curved, like Hilbert spheres, rigidity fails in general. However, in \cite{Song25}, a ``probabilistic'' rigidity phenomenon was found by the third-named author: the induced metric of random equivariant harmonic maps and minimal surfaces in high-dimensional Euclidean spheres tends to concentrate around a unique rescaled hyperbolic metric. 
In this paper, we  substantially strengthen this phenomenon. In particular, our results imply the \emph{uniqueness} of the high probability limit of all the models of random harmonic maps into Euclidean spheres considered in \cite{Song25}.
One of the main ingredients is the new classification of complete totally isotropic minimal surfaces with constant negative curvature in Hilbert spheres. This part is inspired by and unifies earlier contributions of Calabi, Kenmotsu, Bryant \cite{Calabi,Kenmotsu,Bryant} on constant curvature minimal surfaces in Euclidean spheres.

\subsection{Statement of the main results}

Let $\Sigma$ be a closed oriented surface. 
The most important isometric action of $\pi_1(\Sigma)$ (and in fact any discrete group) on a Hilbert sphere is its \emph{regular representation}.
The regular representation $\lambda_{\pi_1(\Sigma)}$  of $\pi_1(\Sigma)$  is the unitary representation on $\ell^2(\pi_1(\Sigma))$ defined as follows: for any $g,x\in \pi_1(\Sigma)$ and $F\in \ell^2(\pi_1(\Sigma))$,
$$(\lambda_{\pi_1(\Sigma)}(g)\cdot F)(x):=F(g^{-1}x).$$
We will consider unitary representations $\rho:\pi_1(\Sigma)\to \End(H)$ which are \emph{weakly equivalent} to $\lambda_{\pi_1(\Sigma)}$.
Weak equivalence is a standard notion of equivalence which is less restrictive than the usual notion of \emph{strong equivalence} for unitary representations\footnote{Two unitary representations are strongly equivalent when they are equal up to conjugation by a linear isometry.}. It will be defined in Section \ref{subsec: weak equivalence}, but roughly speaking, two unitary representations are weakly equivalent if one can approximate arbitrarily well the other. The family of representations weakly equivalent to $\lambda_{\pi_1(\Sigma)}$ is  rich; see \cite{PS86, Szwarc88, KS92, Adams94, Kuhn94, Lukasz16} and Section \ref{subsec: weak equivalence} for examples.

Both our main results will also involve the representation theory of the Lie group $\mathrm{PSL}(2, \mathbb{R})$, which is the orientation preserving isometry group of the hyperbolic plane. Irreducible unitary representations of $\mathrm{PSL}(2, \mathbb{R})$ have been classified by Gelfand, Naimark, Bargmann, Harish--Chandra \cite{GN46, Bargmann47, Chandra52}. We will only consider irreducible unitary representations of $\mathrm{PSL}(2, \mathbb{R})$ of ``class one'', namely those for which there exists a nonzero vector fixed by the compact subgroup $\mathrm{PSO(2)}$. The nontrivial class one representations of $\mathrm{PSL}(2, \mathbb{R})$ are organized into two continuous families called the \textit{principal series} $\{\rho_s\}_{s\in i\mathbb{R}}$ and the \textit{complementary series} $\{\rho_s\}_{s\in (-\frac{1}{2},\frac{1}{2})}$ (see Theorem \ref{thm: irreducible representations}). 
Given any $s\in i\r\cup (-\frac{1}{2},\frac{1}{2})$ and a $\PSO(2)$-invariant unit vector $v$ with respect to $\rho_s$, we denote by $\mathscr{O}_{s,v}$ the orbit of $v$ with respect to $\rho_s$, i.e.
\begin{align}\label{orbit}
    \mathscr{O}_{s,v}:=\{x \mbox{ such that }  x=\rho_s(A)v \mbox{ for some } A\in\PSL(2,\r)\}.
\end{align}
As shown in Section \ref{sec: minimality of special orbits}, each $\mathscr{O}_{s,v}$ is a minimal surface in $\s_H$, with constant negative Gaussian curvature.

Before stating our first theorem, here are some notations. The fundamental group $\pi_1(\Sigma)$ of the surface $\Sigma$ acts by deck transformations on the universal cover $\tilde{\Sigma}$ of $\Sigma$. 
Let $\mathbf{D}_\Sigma\subset \tilde{\Sigma}$ denote any Borel fundamental domain for this action. Let $\chi(\Sigma)$ be the Euler characteristic of $\Sigma$.
Given a Hilbert space $H$, its unit sphere is denoted by $\mathbb{S}_H$ and the Hilbert metric is denoted by $g_H$. For a unitary representation $\rho:\pi_1(\Sigma)\to \End(H)$, $\rho$ induces an isometric action of $\pi_1(\Sigma)$ on the unit sphere $\mathbb{S}_H$.
As above, we define $\mathscr{H}_{\Sigma,\rho}$ as the set of smooth $\pi_1(\Sigma)$-equivariant maps $\tilde{\Sigma}\to \s_H.$ A map in $\mathscr{H}_{\Sigma,\rho}$ is called \textit{full} if its image generates the whole Hilbert space $H$. 
If a unitary representation $\rho$ of $\pi_1(\Sigma)$ factors as
$$\pi_1(\Sigma) \to \mathrm{PSL}(2,\mathbb{R})\to \End(H)$$
where the first arrow is a Fuchsian embedding and the second arrow $\tau$ is a unitary representation of the Lie group $\mathrm{PSL}(2,\mathbb{R})$, we will call such $\rho$ a \emph{$\tau$-Fuchsian} representation of $\pi_1(\Sigma)$ in analogy with the terminology used in higher Teichm\"{u}ller theory.

Our area rigidity theorem is the following:
\begin{theorem}[Area rigidity]\label{thm:area rigidity} 
Let $\Sigma$ be a closed oriented surface of genus $g\ge 2$. Let $\rho:\pi_1(\Sigma)\to \End(H)$ be a unitary representation weakly equivalent to $\lambda_{\pi_1(\Sigma)}$. 
Consider a full immersion $f \in \mathscr{H}_{\Sigma,\rho}$. Then
\begin{align}\label{Equation: fundamental inequality}
    \Area(\mathbf{D}_{\Sigma},f^* g_H) \geq  \frac{\pi}{4}|\chi(\Sigma)|.
\end{align}
Moreover, the following are equivalent:
\begin{enumerate}[(i)]
    \item equality holds: $\Area(\mathbf{D}_{\Sigma},f^* g_H) =  \frac{\pi}{4}|\chi(\Sigma)|$,
    \item $f$ is equivariantly area-minimizing: for any $\hat{f}\in \mathscr{H}_{\Sigma,\rho}$, $\Area(\mathbf{D}_{\Sigma},\hat{f}^* g_H) \geq \Area(\mathbf{D}_{\Sigma},f^* g_H)$, 
    \item $\rho$ is $\rho_0$-Fuchsian, and $f(\tilde{\Sigma}) = \mathscr{O}_{0, v}$ for some $\PSO(2)$-invariant unit vector $v$. 
\end{enumerate}
\end{theorem}

In \cite{Song25}, it was shown that for many models of random equivariant harmonic maps from surfaces to Euclidean spheres, as the dimension of the sphere goes to infinity, the intrinsic geometry of the random harmonic maps converges with high probability to a unique rescaled hyperbolic plane. One of the main motivations for Theorem \ref{thm:area rigidity} is that as a consequence, the \emph{extrinsic geometry} of those harmonic maps also converges to a \emph{unique limit} with high probability. Up to global isometries, this universal limit is of the form $\mathscr{O}_{0, v}$ as in the statement.

Our second main theorem, which is used to establish Theorem \ref{thm:area rigidity} and whose proof forms the technical bulk of the paper, classifies complete totally isotropic minimal surfaces in spheres with constant negative Gaussian curvature as orbits $\mathscr{O}_{s,v}$ defined in (\ref{orbit}). A conformal immersion $u$ from a Riemannian surface to a Hilbert space is called \textit{totally isotropic} if, in every local complex coordinate $z$, we have
\begin{align*}
    \bigg\langle\frac{\partial^pu}{\partial z^p},
    \frac{\partial^qu}{\partial z^q}\bigg\rangle=0
    \qquad\forall\,p,q\geq1,
\end{align*}
where the pairing is the complex bilinear extension of the underlying real inner product of the ambient Hilbert space.
\begin{theorem}[Classification for $K<0$]\label{thm:classification}
A complete connected minimal surface immersed in a Hilbert sphere is totally isotropic and has constant negative Gaussian curvature $K<0$ if and only if, up to a linear isometry between the closed real linear spans of the images, it is equal to $\mathscr{O}_{s,v}$ for some $s\in i\r\cup(-\frac{1}{2},\frac{1}{2})$ such that $K=-\frac{8}{1-4s^2}$, and some $\PSO(2)$-invariant unit vector $v$.
\end{theorem}
\vspace{1em}
\noindent
\textbf{Remarks on Theorem \ref{thm:area rigidity}.}

Theorem \ref{thm:area rigidity} resolves the second half of \cite[Question 9]{Son23b} and is a significant step towards the classification of solutions to the spherical Plateau problem for surfaces \cite[Section 3]{Son23b}. The analogous area rigidity in higher dimensions is conjectured to be true too \cite[Question 8]{Son23b}.

As a consequence of Theorem \ref{thm:area rigidity}, $\rho_0$-Fuchsian representations are the only unitary representations $\rho$ of $\pi_1(\Sigma)$ weakly equivalent to $\lambda_{\pi_1(\Sigma)}$ for which $\mathscr{H}_{\Sigma,\rho}$ contains a full equivariantly area-minimizing immersion. 
In particular, neither the regular representation nor any of the $\rho_s$-Fuchsian representations with $s\in i\r\smallsetminus\{0\}$ admit equivariantly area-minimizing surfaces. 
 
All the results generalize to punctured surfaces $\Sigma$, see Remark \ref{punctured}. The assumption that the map is a full immersion is here to simplify the statement and is not restrictive.

\vspace{1em}
\noindent
\textbf{Remarks on Theorem \ref{thm:classification}.}

The study of constant Gaussian curvature minimal surfaces in Euclidean spheres is classical. The seminal paper of Calabi \cite{Calabi} on constant curvature minimal spheres in Euclidean spheres provides a full classification: such minimal spheres are exactly given by immersions whose coordinates are a suitable basis of normalized spherical harmonics corresponding to a given eigenvalue.
Kenmotsu took up the classification of flat minimal surfaces in spheres in \cite{Kenmotsu}. In both the round and flat cases, the constant curvature minimal surfaces turn out to be orbits of certain unitary representations of the isometry group of the round 2-sphere or the flat 2-plane. Bryant then showed in \cite{Bryant} that there are no minimal surfaces in Euclidean spheres with constant negative Gaussian curvature, and also strengthened \cite{Calabi,Kenmotsu}. The methods of those three papers, while related, are different. Our proof of Theorem \ref{thm:classification} is mostly based on the strategy of Bryant \cite{Bryant}, although we also combine ideas from the earlier papers (see Section \ref{sec: comments on the proofs}).

In contrast to previous work on the topic, the ``totally isotropic'' condition in Theorem \ref{thm:classification} is necessary: there are non-standard examples of (non-totally isotropic) hyperbolic minimal surfaces in infinite dimensional Hilbert spheres, see Example \ref{Example: necessity of total isotropy}! This example was found while the first-named author was holding a reading seminar on the topic at the University of Warwick, following the realization that there was a gap in the proof of Proposition~3.4 in the older version of this paper. To resolve this issue, we added the total isotropy assumption in the statement of Theorem~\ref{thm:classification}, and Lemma~\ref{Lemma: total isotropy in the equality case} was added to obtain the original area rigidity statement (Theorem~\ref{thm:area rigidity}).

%

\vspace{1em}
\noindent
\textbf{Further questions.}

Theorems \ref{thm:area rigidity} and \ref{thm:classification} suggest several directions for further investigation. The $\mathrm{PSL}(2,\mathbb{R})$-orbits $\mathscr{O}_{s,v}$ appearing in Theorem \ref{thm:classification} are minimal surfaces in Hilbert spheres which are both canonical and explicit. It is thus surprising that not much is known about their variational properties. As of now, the following is known: $\mathscr{O}_{0,v}$ is area-minimizing with respect to compact variations or equivariant variations for $\rho_0$-Fuchsian representations by \cite{BCG95}; 
for $s\in i\mathbb{R}\smallsetminus \{0\}$, $\mathscr{O}_{s,v}$ is not equivariantly area-minimizing with respect to $\rho_s$-Fuchsian representations of a surface group by \cite{Song25} or Theorem \ref{thm:area rigidity}. \emph{Which orbits $\mathscr{O}_{s,v}$ are area-minimizing with respect to compact variations or equivariant variations for $\rho_s$-Fuchsian representations?} 
It seems plausible that all $\mathscr{O}_{s,v}$ are area-minimizing when $s\in (-\frac{1}{2},\frac{1}{2})$. Which orbits are stable?

In a sense, Theorems \ref{thm:area rigidity} and \ref{thm:classification} confirm in a special case a heuristic appearing in \cite{Song25} that in some situations, being area-minimizing becomes a more constraining condition as the ambient dimension gets large, so that equivariantly area-minimizing surfaces in infinite dimensions should be special. For example, equivariant area-minimizing surfaces in spheres with respect to ``self-similar'' representations are shown to have unique intrinsic geometry by \cite[Theorem 2.1]{Song25}. \emph{Do these also have unique extrinsic geometry}, as in the case of Theorem \ref{thm:area rigidity}? 

Theorem \ref{thm:area rigidity} identifies a very restricted family of representations, $\rho_0$-Fuchsian representations, as the best representations of a surface group which are weakly equivalent to the regular representation. To what extent does this picture generalize to representations far from the regular representation? In other words, \emph{what can be said about the space of unitary representations of surface groups, which admit an equivariant area-minimizing map into the Hilbert sphere?} There are related conjectures (rigidity and existence) in higher dimensions, see \cite[Question 8]{Son23b} and \cite[Conjecture 1.2]{Song24}.


%

\subsection{Comments on the proofs and relation to prior work}\label{sec: comments on the proofs}

To prove Theorem \ref{thm:area rigidity}, we rely on \cite{Song25}.
Since the topological lower bound \eqref{Equation: fundamental inequality} is independent of the chosen representation $\rho$, (i) easily implies (ii). Nevertheless, the converse is non-trivial and depends on Proposition \ref{Proposition: weak equivalence implies same spherical area}, which states that all unitary representations of a given surface group, which are weakly equivalent to the regular representation, have the same ``spherical area''. (iii) implies (i) by an explicit calculation. The most interesting implication is from (i) to (iii).
For this last implication, we use \cite[Corollary 2.4]{Song25} where it was shown that if a minimal immersion $\psi\in \mathscr{H}_{\Sigma,\rho}$ is equivariant with respect to a representation $\rho$ of a surface group which is weakly equivalent to the regular representation, and if this map is energy-minimizing with respect to a given conformal structure on $\Sigma$, then the pullback metric on $\Sigma$ by $\psi$ has constant Gaussian (i.e. sectional) curvature $-8$. We then prove that equality in the energy bound implies total isotropy, by realizing the representation inside a representation of the reduced crossed product associated with the boundary action and completing the square in the energy (Lemma \ref{Lemma: total isotropy in the equality case}). By the classification theorem, Theorem \ref{thm:classification}, the minimal surface $\psi(\tilde{\Sigma})$ is necessarily an orbit of the form $\mathscr{O}_{0,v}$, which has Gaussian curvature $-8$ (up to an ambient isometry of the Hilbert sphere).


\vspace{1em}

\noindent
Next, we aim to outline here the main ideas and the structure of the proof of Theorem \ref{thm:classification}. 
%
\begin{enumerate}[(1)]
    \item First, we check that if $v$ is a $\PSO(2)$-invariant vector, then the orbit $\mathscr{O}_{s,v}$ as in (\ref{orbit}) is indeed a minimal surface immersed in the Hilbert sphere $\mathbb{S}_H$, of constant negative curvature $K=-\frac{8}{1-4s^2}$. This is done in Theorem \ref{thm: special orbit is minimal}. This proves the ``if'' part of the statement. We move on to the ``only if'' part.
    \item  Let $u:\mathbb{H}^2\to\mathbb{S}_H$ be a totally isotropic minimal isometric immersion.
    As Bryant did in  \cite{Bryant}, we lift $u$ to the frame bundle $\pi:\mathscr{F}\cong\operatorname{PSL}(2,\mathbb{R})\to\mathbb{H}^2$ of the hyperbolic plane $\mathbb{H}^2$ by letting $u_0:=\pi^*u\in C^{\infty}(\mathscr{F},\mathbb{S}_H)$ and we study the algebra of the holomorphic and anti-holomorphic differentials of $u_0$ (see Section \ref{subsec: differentials}).\footnote{As usual in moving frames theory, we will often not distinguish between $u$ and $u_0$.} 
    The total isotropy assumption allows us to compute all pairings between holomorphic and anti-holomorphic differentials of $u$ (Proposition \ref{prop: differential of R_m}).      
    This should be compared with the proof of  \cite[Theorem 1.6]{Bryant} for minimal isometric immersions of spheres.
    
    %
    \item Second, as in \cite{Calabi,Kenmotsu}, we define the osculating space $\mathbf{O}(q,u)\subset H\otimes_{\r}\c$ of $u$ at some point $q\in\mathscr{F}$ to be the closure of the span of the holomorphic and anti-holomorphic differentials of $u_0$ of any order at $q$. 
    The next step is to prove that $\mathbf{O}(q,u)$ is independent of $q\in\mathscr{F}$ (Proposition \ref{Proposition: the osculating space of a minimal isometric immersion is trivial}). In the finite dimensional case, this is essentially an easy consequence of the Leibniz formula (see \cite[Proof of Theorem 1.6, page 264]{Bryant}), whilst in infinite dimensions we need to show that $u$ is analytic and leverage a unique continuation principle. 
    Thus, we prove that
    $\mathbf{O}(q,u)= \mathbf{T}(u)$ for some fixed closed subspace $\mathbf{T}(u)\subset H\otimes_{\r}\c$ and, by construction, we have $u(\mathbb{H}^2)=u_0(\mathscr{F})\subset \mathbf{T}(u)$.
    %
    \item In analogy with \cite{Bryant}, we define linear operators $\{F_1,F_2,F_3\}$ on $\mathbf{T}(u)$. We show in our case that they are essentially skew-adjoint and that they generate a Lie algebra isomorphic to $\mathfrak{sl}(2,\mathbb{R})$. To establish these facts requires nontrivial work in infinite dimension, because the operators are unbounded. 
    Next, the operators $\{F_1,F_2,F_3\}$ induce a Lie algebra representation $\varphi$ of $\mathfrak{sl}(2,\mathbb{R})$ on $\mathbf{T}(u)$ by essentially skew-adjoint operators. 
    We then check that $\varphi$ can be lifted to a unitary irreducible representation $\rho$ of $\operatorname{PSL}(2,\mathbb{R})$ on $\mathbf{T}(u)$. Lastly, we prove that $u(\mathbb{H}^2)$ must be an orbit of $\rho$. The statement then follows from Theorem \ref{thm: irreducible representations} and Theorem \ref{thm: special orbit is minimal}.
\end{enumerate}

\vspace{1em}
\subsection*{Acknowledgements}
We are grateful to Robert Bryant for answering our questions about his work and to Christine Breiner for discussions.

X. L. would like to thank his advisor Xin Zhou for his constant support. Part of this work was done when X. L. visited Caltech and he is grateful for their hospitality. X. L. was partially supported by Simons Dissertation Fellowship, NSF grants DMS-2104254 and DMS-2243149, and NSF grant DMS-2452599 for the IHES Summer School on Discrete Subgroups of Lie Groups: Dynamics, Actions, Rigidity.

A. S. would like to thank Ursula Hamenst\"{a}dt for discussions. A. S. was partially supported by NSF grant DMS-2405175 and an Alfred P. Sloan Research Fellowship. 

R. C. would like to thank Renato Ghini Bettiol and Yannick Sire for discussions and to Yuxuan Song for inquiring about the proof of Proposition~3.4 in the older version of this paper, which eventually led R. C. to spot the gap in its argument. R. C. was partially supported by an AMS-Simons Travel Grant.

\vspace{1em}
\subsection*{AI statement:}
An LLM was used for the proof of Lemma \ref{Lemma: total isotropy in the equality case}. This was added to correct a first version of this paper which appeared in the summer of 2025.

\section{Area rigidity of \texorpdfstring{$\rho_0$}{Z}-Fuchsian representations}\label{sec: area rigidity} 
In this section, we prove Theorem \ref{thm:area rigidity} assuming the classification result, Theorem \ref{thm:classification}, which will be proved in the next sections.

\subsection{Preliminaries on weak equivalence and the regular representation}\label{subsec: weak equivalence} 

Throughout the paper, all Hilbert spaces of concern will be separable, with finite or infinite dimensions, and complex unless otherwise noted. Let $\Gamma$ be a finitely generated group. We say that $(\pi, H)$ is a \textit{unitary representation} of $\Gamma$ if $H$ is a Hilbert space, and $\pi$ is a group morphism from $\Gamma$ to the unitary group $\operatorname{U}(H) \subset \End(H)$ of $H$.
Our convention will be that $\Gamma$ is called a surface group if it is the fundamental group of a closed oriented  surface with possibly finitely many punctures removed.


Two unitary representations are strongly equivalent if they are equal after conjugation by a linear isometry of the two Hilbert spaces \cite[Definition A.1.3]{BdLHV08}. 
This notion of equivalence is far too rigid for representations of noncompact groups like surface groups. The ``right'' notion of equivalence is that of weak equivalence, which is central in this paper.
 \begin{definition} \label{weak equiv}
 Let $(\pi,H)$ and $(\rho,K)$ be two unitary representations of $\Gamma$. We say that $\pi$ is weakly contained in $\rho$ if for every $\xi \in H$, every finite subset $Q$ of $\Gamma$, and every $\varepsilon>0$, there exist $\eta_1,...,\eta_m$ in $K$ such that for all $g\in Q$,
 \begin{align*}
     \bigg\lvert\langle \pi(g)\xi,\xi \rangle - \sum_{j=1}^m \langle \rho(g)\eta_j,\eta_j\rangle\bigg\rvert <\varepsilon.
 \end{align*}
 We write $\pi\prec\rho$ if the above holds. 
 If we have both  $\pi\prec\rho$ and $\rho\prec\pi$, then we say that $\pi$ and $\rho$ are weakly equivalent, and we write $\pi\sim\rho$.
 \end{definition}
Recall that the regular representation $\lambda_\Gamma$ of the group $\Gamma$ is the standard $\Gamma$-action on $\ell^2(\Gamma)$: more precisely,
for any $g,x\in \Gamma$ and $F\in \ell^2(\Gamma)$, 
$$(\lambda_\Gamma(g)\cdot F)(x) :=F(g^{-1}x).$$
As for finite groups, this is the most important representation of $\Gamma$. Unlike finite groups however, $\lambda_\Gamma$ is not completely understood in general.

From now on, assume that $\Gamma$ is a surface group.
If $\Gamma$ is non-abelian, then it is known \cite{Harpe85} that $\Gamma$ is $C^*$-simple: in particular, any unitary representation of $\Gamma$ weakly contained in $\lambda_\Gamma$ is actually weakly equivalent to $\lambda_\Gamma$. See the survey \cite{Harpe07} for more details on $C^*$-simplicity. 

There is a wide variety of unitary representations weakly equivalent to $\lambda_\Gamma$. One concrete but rich family, called \textit{boundary representations}, can be described as follows. Consider $\partial \Gamma$ the Gromov boundary of $\Gamma$ and let $\mu$ be any Patterson--Sullivan measure on $\partial \Gamma$ induced by a metric $d$ on $\Gamma$. Then $\Gamma$ acts on the measure space $(\partial \Gamma, \mu)$ in a way that preserves the measure class of $\mu$. By a standard recipe, this provides a unitary representation of $\Gamma$ on the space of $L^2$-functions:
$$\pi_{d,\mu}: \Gamma\to \End(L^2(\partial\Gamma,\mu)).$$
Any such representation is weakly equivalent to the regular representation $\lambda_\Gamma$ \cite[Theorem 5.1]{Adams94}. Moreover, by \cite[Theorem 7.4]{Lukasz16}, two such representations $\pi_{d,\mu}, \pi_{d',\mu'}$ are strongly equivalent 
if and only if the metrics $d,d'$ are roughly similar, which is a very restrictive condition.

Another more general family of examples is given by \textit{quasi-regular representations} \cite[Definition A.6.1]{BdLHV08}.
If $\Gamma$ acts on any quasi-regular measure space $(X,\mu)$ in a way that preserves the measure class of $\mu$, then we obtain a corresponding unitary representation
$$\pi_\mu: \Gamma\to \End(L^2(X,\mu)).$$
In \cite{Kuhn94}, the following criterion is proved: if the action of $\Gamma$ is ergodic and ``amenable'', then $\pi_\mu$ is weakly equivalent to $\lambda_\Gamma$.  

\subsection{Area and energy of unitary representations of surface groups} 
Recall that for a smooth map $u$ from a closed oriented Riemannian surface $(\Sigma, g_{\Sigma})$ of genus $g \geq 2$ to a smooth Hilbert manifold $(M, g_M)$, the energy of $u$ restricted to a subset $D \subset \Sigma$ is defined as  
\begin{align*}
    \E(u|_{D}) = \frac{1}{2} \int_{D} |du(x)|^2 dv_{g_{\Sigma}}(x).  
\end{align*}

Given a separable Hilbert space $H$, let $\mathbb{S}_H$ be its unit sphere, endowed with the Riemannian metric $g_H$ induced by the inner product of $H$.
For a unitary representation $\rho: \pi_1(\Sigma) \rightarrow \End(H)$,  $\rho$ induces an isometric action of $\pi_1(\Sigma)$ on the unit sphere $\mathbb{S}_H$ of $H$. As usual, we consider the action of  $\pi_1(\Sigma)$ by deck transformations on the universal cover $\tilde\Sigma$ of $\Sigma$. Consider the space of maps 
\begin{align*}
    \mathscr{H}_{\Sigma,\rho} := \{&\pi_1(\Sigma)\text{-equivariant smooth maps $f: \tilde{\Sigma} \rightarrow \mathbb{S}_{H}$ with respect to } \rho\}.  
\end{align*} 
%

Let $\mathbf{D}_{\Sigma}$ be a Borel fundamental domain in $\tilde{\Sigma}$ with piecewise smooth boundary. 
Denote by $\mu_0$ a conformal structure on $\Sigma$, namely a point in the Teichm\"{u}ller space $\mathscr{T}_{\Sigma}$ of $\Sigma$, and let $g_0$ be the unique hyperbolic metric on $\Sigma$ that represents $\mu_0$. As in \cite{Song25}, we define the spherical area and energy of a unitary representation of $\pi_1(\Sigma)$ with respect to $\rho$ as follows.  
\begin{definition}\label{defae}
The  energy of $(\Sigma, \mu_0, \rho)$ is defined as 
\begin{align*}
    \E(\Sigma, \mu_0, \rho) := \inf \{\E(u|_{(\mathbf{D}_{\Sigma}, g_0)}) \mbox{ : } u \in \mathscr{H}_{\Sigma,\rho}\}. 
\end{align*}
The area of $(\Sigma, \rho)$ is defined as 
\begin{align*}
    \Area(\Sigma, \rho) := \inf \{\E(\Sigma, \mu_0, \rho) \mbox{ : } \mu_0 \in \mathscr{T}_{\Sigma}\}. 
\end{align*}
\end{definition}

\begin{remark}\label{rem1}
\phantom{.}
\begin{itemize}
\item 
Recall the basic fact that the energy of a map is conformally invariant, in the sense that it does not depend on the metric in the conformal class of $g_0$.

\item It is well-known that the following holds for any
$V\in\mathscr{H}_{\Sigma,\rho}$ by Cauchy-Schwarz:
    $$\Area(\mathbf{D}_\Sigma, V^*g_H) \leq \E(V|_{(\mathbf{D}_{\Sigma},g_\Sigma)})$$
    and equality holds when $V$ is weakly conformal.

    \item Conversely, since $\Sigma$ is a closed surface, for any
    $V\in\mathscr{H}_{\Sigma,\rho}$ and any $\varepsilon>0$,
    there is a metric $g_\Sigma$ on $\Sigma$ such that
    $$\Area(\mathbf{D}_\Sigma, V^*g_H) \geq \E(V|_{(\mathbf{D}_{\Sigma},g_\Sigma)}) -\varepsilon.$$
    This is standard, see for instance the proof of \cite[Lemma 4.4 in Chapter 4]{CM11}.
    \item Hence,    the spherical area of $(\Sigma, \rho)$ satisfies 
    \begin{align*}
          \Area(\Sigma, \rho) = \inf\{\Area(\mathbf{D}_{\Sigma}, u^{*}g_H) \mbox{ : } u \in \mathscr{H}_{\Sigma,\rho}\} < \infty. 
    \end{align*}
    \end{itemize}
\end{remark}

%
%

The proof of Theorem \ref{thm:area rigidity} is crucially based on the following intrinsic energy rigidity statement, proved in 
\cite[Corollary 2.4]{Song25}.

\begin{theorem}\label{thm: intrinsic energy rigidity theorem} 
Let $g_0$ denote a hyperbolic metric on $\Sigma$ and also its lift to $\tilde{\Sigma}$.
Let $\rho: \pi_1(\Sigma) \rightarrow \End(H)$ be a unitary representation of $\pi_1(\Sigma)$ which is weakly equivalent to $\lambda_{\pi_1(\Sigma)}$. Consider a smooth $\pi_1(\Sigma)$-equivariant map with respect to $\rho$
\begin{align*}
    f:\tilde{\Sigma}\to \mathbb{S}_{H}.
\end{align*}
Then
\begin{align*}
    \E(f|_{(\mathbf{D}_{\Sigma}, g_0)}) \geq \Area(\Sigma, \lambda_{\pi_1(\Sigma)}) = \frac{\pi}{4}|\chi(\Sigma)|
\end{align*}
and equality holds if and only if 
$$f^{*} g_{H} = \frac{1}{8} g_0.$$
%
\end{theorem}

The next proposition states that unitary representations weakly equivalent to the regular representation all have the same spherical area (whether this holds for other weak equivalence classes is still open):
\begin{proposition}\label{Proposition: weak equivalence implies same spherical area}
    Let $\rho:\pi_1(\Sigma)\to \End(H)$ be a unitary representation weakly equivalent to $\lambda_{\pi_1(\Sigma)}$. Then 
    \begin{align*}
        \Area(\Sigma,\rho)=\Area(\Sigma,\lambda_{\pi_1(\Sigma)})=\frac{\pi}{4}\lvert\chi(\Sigma)\rvert.
    \end{align*}
\end{proposition}
\begin{proof}
Let $[g_0]$ denote the conformal class of the hyperbolic metric $g_0$. 

\noindent
Consider the boundary representation $$\underline{\rho}_B : \pi_1(\Sigma) \to \End(L^2(\partial \tilde\Sigma))$$ defined in  \cite[(7)]{Song25}.
By \cite[Lemma 1.14(1)]{Song25}, the boundary representation $\underline{\rho}_B$ is irreducible and weakly equivalent to $\lambda_{\pi_1(\Sigma)}$. In particular since $\rho\sim \lambda_{\pi_1(\Sigma)}$ by assumption, we have \begin{equation}\label{wc}
\underline{\rho}_B\prec\rho
\end{equation}with $\underline{\rho}_B$ irreducible. Identify $(\tilde{\Sigma},g_0)$ with the Poincar\'e disk
$(\D,4|dz|^2/(1-|z|^2)^2)$, taking the basepoint used to define $\underline{\rho}_B$ to $0$. The boundary measure is then the normalized arc length measure $\mu=d\theta/(2\pi)$ on $\s^1$. Define the Poisson kernel $P$ of $\mathbb{D}$ and the map $\mathscr{P}:\D\to\s_{L^2(\s^1,\mu\,;\c)}$ by
\begin{align}\label{Equation: Poisson map}
    P(z,\xi):=\frac{1-|z|^2}{|\xi-z|^2}, \qquad \mathscr{P}(z):=\sqrt{P(z,\,\cdot\,)}\quad\forall\,z\in\D.
\end{align}
By \cite[Lemma 1.14(2)]{Song25}, $\mathscr{P}$ is a $\underline{\rho}_B$-equivariant isometric embedding of $(\D,\frac{1}{8}g_0)$ into the unit sphere of $L^2(\s^1,\mu\,;\c)$. In particular, its energy on $(\mathbf{D}_\Sigma,g_0)$ is equal to $\Area(\Sigma,\lambda_{\pi_1(\Sigma)})$. By Definition \ref{defae}, we then have
\begin{equation}\label{in en ar}
    \E(\Sigma,[g_0],\underline{\rho}_B) \le \Area(\Sigma,\lambda_{\pi_1(\Sigma)}).
\end{equation}
We claim that
\begin{align*}
    \Area(\Sigma,\rho) & \le \E(\Sigma,[g_0],\rho)\\
    &\le \E(\Sigma,[g_0],\underline{\rho}_B) \\
    &\le \Area(\Sigma,\lambda_{\pi_1(\Sigma)})\\
    &=\frac{\pi}{4}\lvert\chi(\Sigma)\rvert.
\end{align*}
Indeed, the first inequality follows from  Definition \ref{defae}, the second inequality is (\ref{wc}) combined with \cite[Corollary 3.4(2)]{Song25}\footnote{In \cite{Song25}, one considers punctured surfaces and the ``renormalized energy'' $\E_{\mathrm{ren}}$ of maps. In this paper, we only work with closed surfaces, and so $\E_{\mathrm{ren}}$ is simply the usual energy $E$.}, the third inequality is (\ref{in en ar}), and the last line is a theorem of Besson-Courtois-Gallot \cite[Theorem 1.12]{Song25}.

\noindent
It remains to show that $ \Area(\Sigma,\rho)\geq \frac{\pi}{4}\lvert\chi(\Sigma)\rvert$. We claim that 
\begin{align*}
    \frac{\pi}{4}\lvert\chi(\Sigma)\rvert & =\Area(\Sigma,\lambda_{\pi_1(\Sigma)})\\
    &=\E\bigg(\Sigma,[g_0],\bigoplus^{\infty}\lambda_{\pi_1(\Sigma)}\bigg)\\
    &\le\E(\Sigma,[g_0],\rho),
\end{align*}
Indeed,  the first line is \cite[Theorem 1.12]{Song25}, the second line is \cite[Lemma 1.13 (second equality)]{Song25}, and the last inequality comes from our assumption that $\rho\sim\lambda_{\pi_1(\Sigma)}$ combined with \cite[Corollary 3.4(1)]{Song25}. 
Here, $\bigoplus^{\infty}\lambda_{\pi_1(\Sigma)}$ denotes the infinite direct sum of $\lambda_{\pi_1(\Sigma)}$. The general construction is defined in \cite[Definition A.1.6]{BdLHV08}. Now, since the above inequality holds for any conformal class $[g_0]$, by Definition \ref{defae}, we actually get
\begin{align*}
    \frac{\pi}{4}\lvert\chi(\Sigma)\rvert \leq \Area(\Sigma,\rho).
\end{align*}
The desired statement follows.
\end{proof}
We next show that equality in the energy bound implies total isotropy.
\begin{lemma}\label{Lemma: total isotropy in the equality case}
    Under the assumptions of Theorem \ref{thm: intrinsic energy rigidity theorem}, suppose that
    \begin{align*}
        \E(f|_{(\mathbf{D}_{\Sigma},g_0)})
        =\frac{\pi}{4}|\chi(\Sigma)|.
    \end{align*}
    Then $f$ is totally isotropic. Moreover, its Hermitian inner products satisfy
    \begin{align*}
        \langle f(x),f(y)\rangle_H\in\r
        \qquad\forall\,x,y\in\tilde{\Sigma}.
    \end{align*}
\end{lemma}
\begin{proof}
    Identify $(\tilde{\Sigma},g_0)$ with the Poincar\'e disk and let $\Gamma<\PSL(2,\r)$ be its group of deck transformations. Let $P$ and $\mathscr{P}$ be as in \eqref{Equation: Poisson map}. Let $\mathcal{A}:=C(\s^1)\rtimes_r\Gamma$ be the reduced crossed product associated with the boundary action of $\Gamma$, as in Definition \ref{Definition: reduced group algebra and crossed product}. Since $\rho\prec\lambda_\Gamma$, $\rho$ extends by linearity and continuity to a unital $*$-representation of $C_r^*(\Gamma)\subset\mathcal{A}$ on $H$. By Arveson's extension theorem and Stinespring's dilation theorem \cite[Theorems 1.2.3 and 1.1.1]{Arveson69}, there exist a Hilbert space $\hat H$, a linear isometry $V:H\to\hat H$ and a unital $*$-representation $\sigma:\mathcal{A}\to\mathcal{B}(\hat H)$ such that
    \begin{align*}
        V^*\sigma(\gamma)V=\rho(\gamma)\qquad\forall\,\gamma\in\Gamma.
    \end{align*}
     Hence, $\sigma(\gamma)V=V\rho(\gamma)$ for every $\gamma\in\Gamma$.
    The crossed-product relation gives
    \begin{align}\label{Equation: boundary covariance}
        \sigma(\gamma)\sigma(a)\sigma(\gamma)^{-1}=\sigma(a\circ\gamma^{-1})\qquad\forall\,\gamma\in\Gamma,\quad\forall\,a\in C(\s^1).
    \end{align}
    %
    %

    \noindent
    Let $u:=Vf$ and define the operator-valued $1$-form $A$ on $\D$ by
    \begin{align*}
        A_z(X):=\frac{1}{2}\sigma\big(d_z\log P(z,\,\cdot\,)(X)\big) \qquad\forall\,z\in\D,\quad\forall\,X\in T_z\D.
    \end{align*}
    Since $\sigma$ preserves adjoints and $P(z,.)$ is a real-valued and smooth function, the operators $A_z(X)$ are bounded and self-adjoint. Direct computation gives
    \begin{align*}
        |d\log P(\cdot,\xi)|_{g_0}^2=1,
        \qquad
        \Delta_{g_0}\log P(\cdot,\xi)=-1.
    \end{align*}
    Thus, for every local $g_0$-orthonormal frame $e_1,e_2$,
    \begin{align}\label{Equation: boundary operator identities}
        \sum_{j=1}^2A(e_j)^2=\frac{1}{4}\operatorname{Id}_{\hat H}, \qquad\divergence_{g_0}A=-\frac{1}{2}\operatorname{Id}_{\hat H}.
    \end{align}
    Here the divergence is the trace of the covariant derivative, using the trivial connection on the operator bundle.

    \noindent
    The Poisson kernel satisfies
    \begin{align*}
        P(\gamma z,\xi)=P(z,\gamma^{-1}\xi)|(\gamma^{-1})'(\xi)|.
    \end{align*}
    Taking the log on both sides, differentiating in $z$ and using \eqref{Equation: boundary covariance}, we obtain
    \begin{align*}
        A_{\gamma z}(d\gamma X)=\sigma(\gamma) A_z(X)\sigma(\gamma)^{-1}.
    \end{align*}
    Recall that $du$ and $Au$ are functions from $T_z \mathbb{D}$ to $\hat{H}$.
    Since $u(\gamma z)=V\rho(\gamma)f(z)=\sigma(\gamma)u(z)$ for every $\gamma\in\Gamma$ and $z\in\D$, the preceding transformation formula implies that $|du-Au|_{g_0}^2$ and the real $1$-form
    \begin{align*}
        \alpha(X):=\Real\langle A(X)u,u\rangle_{\hat H}
    \end{align*}
    descend to $\Sigma$. Since $\|u\|_{\hat H}=1$, \eqref{Equation: boundary operator identities} gives
    \begin{align*}
        \divergence_{g_0}(\alpha^\sharp)&=-\frac{1}{2}+2\sum_{j=1}^2\Real\langle du(e_j),A(e_j)u\rangle_{\hat H},\\
        |du-Au|_{g_0}^2&=|du|_{g_0}^2-\frac{1}{4}-\divergence_{g_0}(\alpha^\sharp).
    \end{align*}
    Integrating over $\Sigma$ and using Gauss--Bonnet, we conclude that
    \begin{align}\label{Equation: boundary energy identity}
        \frac{1}{2}\int_\Sigma |du-Au|_{g_0}^2\,dv_{g_0} = \E(f|_{(\mathbf{D}_\Sigma,g_0)})-\frac{\pi}{4}|\chi(\Sigma)|.
    \end{align}
    Our equality assumption therefore implies $du=Au$.

    \noindent
    Since the operators in $\sigma(C(\s^1))$ commute, since $\sigma$ is a unital $*$-representation, differentiating and using the definition of $A$ and the equation $du=Au$ gives:
    \begin{align*}
        d\big(\sigma(\mathscr{P}(z))^{-1}u(z)\big)=0.
    \end{align*}
    Thus, setting $v:=u(0)$, we have
    \begin{align}\label{Equation: integrated boundary equation}
        u(z)=\sigma(\mathscr{P}(z))v \qquad\forall\,z\in\D.
    \end{align}
    Let $\ell:C(\s^1)\to\c$ be the positive linear functional given by $\ell(a):=\langle\sigma(a)v,v\rangle_{\hat H}$. Since $V$ is a linear isometry, since $\mathscr{P}$ is real-valued, and since $\sigma$ is a $*$-representation, \eqref{Equation: integrated boundary equation} gives
    \begin{align}\label{Equation: inner products in the equality case}
        \langle f(z),f(w)\rangle_H=\ell(\mathscr{P}(z)\mathscr{P}(w))\in\r \qquad\forall\,z,w\in\D.
    \end{align}
    In particular, taking $w=z$, we obtain \begin{equation}\label{bloo}
    \ell(P(z,\,\cdot\,))=\|f(z)\|_H^2=1 \quad \text{for every $z\in\D$}.
    \end{equation}
    
    \noindent
    For every $n\geq1$, set $a_n(\xi):=\xi^{-n}$ for $\xi\in\s^1$ and $b_n:=\prod_{j=0}^{n-1}(j+\frac{1}{2})$. Since
    \begin{align*}
        P(z,\xi)=\frac{1-z\overline{z}}{(1-z\xi^{-1})(1-\overline{z}\xi)},
    \end{align*}
    the pure holomorphic derivatives at $0$ satisfy
    \begin{align*}
        \frac{\partial^n P}{\partial z^n}(0,\,\cdot\,)=n!\,a_n(\cdot),
        \qquad
        \frac{\partial^n\mathscr{P}}{\partial z^n}(0)=b_n a_n.
    \end{align*}
    The inner product in \eqref{Equation: inner products in the equality case} is real-valued, so it agrees with the underlying real inner product of $f(z)$ and $f(w)$. Hence, differentiating $p$ times in $z$ and $q$ times in $w$ and setting $z=w=0$, we obtain, for the complex bilinear extension of this real inner product,
    \begin{align*}
        \bigg\langle\frac{\partial^p f}{\partial z^p}(0),\frac{\partial^q f}{\partial z^q}(0)\bigg\rangle&=b_p b_q\,\ell(a_p a_q)=b_p b_q\,\ell(a_{p+q})\quad \text{ (use the definition $a_n(\xi)=\xi^{-n}$)}\\
        &=\frac{b_p b_q}{(p+q)!}\left.\frac{\partial^{p+q}}{\partial z^{p+q}}\|f(z)\|_H^2\right|_{z=0} \qquad\forall\,p,q\geq1\quad \text{(use the first equality in (\ref{bloo}))}\\
        & =0 \quad \text{(use the second equality in (\ref{bloo}))}.
    \end{align*}
    Since our initial identification of $\tilde{\Sigma}$ with $\D$ can be chosen to send any prescribed point to $0$, and the vanishing of these pairings is preserved under holomorphic changes of coordinates, $f$ is totally isotropic.
\end{proof}
\subsection{Proof of Theorem \ref{thm:area rigidity}} 
The main inequality in the statement is a consequence of Theorem \ref{thm: intrinsic energy rigidity theorem}, Definition \ref{defae} and Remark \ref{rem1}.

\noindent
Next, we show separately that $(i)$ is equivalent to both $(ii)$ and $(iii)$.

\smallskip
\noindent
$(i)\Leftrightarrow (ii)$. Notice that, by Definition \ref{defae}, Remark \ref{rem1}  and Proposition \ref{Proposition: weak equivalence implies same spherical area}, we have
\begin{align*}
    \Area(\mathbf{D}_{\Sigma},\hat f^*g_H)\ge\Area(\Sigma,\rho)=\Area(\Sigma,\lambda_{\pi_1(\Sigma)})=\frac{\pi}{4}\lvert\chi(\Sigma)\rvert \qquad\forall\,\hat f\in\mathscr{H}_{\Sigma,\rho}.
\end{align*}
From this inequality, it is easy to check that $f$ satisfies $(i)$ if and only if it satisfies $(ii)$.

\medskip
\noindent
$(i)\Leftrightarrow (iii)$. The fact that $(iii)$ implies $(i)$ follows by a direct computation done in Theorem \ref{thm: special orbit is minimal}. This theorem says that the 2-dimensional surface $f(\tilde{\Sigma}) = \mathscr{O}_{0,v}$ endowed with the induced metric has Gaussian curvature $-8$. The area of $\mathbf{D}_\Sigma$ with the pullback metric by $f$ is then equal to the area of the rescaled hyperbolic metric  on $\Sigma$ with curvature $-8$, namely $\frac{\pi}{4}|\chi(\Sigma)|$.

\noindent
We are just left to prove that $(i)$ implies $(iii)$. Let then $f \in \mathscr{H}_{\Sigma,\rho}$ be a full immersion such that
\begin{align*}
    \Area(\mathbf{D}_{\Sigma},f^* g_H)=\frac{\pi}{4}|\chi(\Sigma)|.
\end{align*}
Endow $\Sigma$ with the conformal class induced by the pullback metric $f^* g_H$, so that $f$ is now a conformal map. The energy of $f$ restricted to  $\mathbf{D}_\Sigma$  is then equal to $\Area(\mathbf{D}_{\Sigma},f^* g_H)$ (see Remark \ref{rem1}).
Again by Theorem \ref{thm: intrinsic energy rigidity theorem}, we conclude that $f^*g_H=\frac{1}{8}g_0$. Moreover, since $f$ is equivariantly area-minimizing by $(ii)$, we deduce that
\begin{align*}
    f:\bigg(\tilde{\Sigma},\frac{1}{8}g_0\bigg)\to\mathbb{S}_H
\end{align*}
is a minimal isometric immersion. By Lemma \ref{Lemma: total isotropy in the equality case}, $f$ is totally isotropic. Let
\begin{align*}
    H_{\mathbb{R}}
    :=\overline{\operatorname{span}_{\mathbb{R}}
    \{f(x)\mbox{ : }x\in\tilde{\Sigma}\}}\subset H.
\end{align*}
Since $f$ is totally isotropic, the Hermitian inner product on $H_{\mathbb{R}}$ is real. This implies that $H_{\mathbb{R}}$ and $iH_{\mathbb{R}}$ are orthogonal for the underlying real inner product. Their sum is therefore closed, and linear fullness of $f$ gives
\begin{align*}
    H=H_{\mathbb{R}}\oplus iH_{\mathbb{R}}.
\end{align*}
Thus, the canonical map $H_{\mathbb{R}}\otimes_{\mathbb{R}}\mathbb{C}\to H$ is unitary. Applying Theorem \ref{Theorem: classification negative curvature} to $f$ with values in $\mathbb{S}_{H_{\mathbb{R}}}$, and recalling that the induced metric has Gaussian curvature $-8$, we deduce that, after a unitary identification of $H$,
\begin{align*}
    f(\tilde{\Sigma})=\mathscr{O}_{0,v}
\end{align*}
for some $\PSO(2)$-invariant unit vector $v$. Moreover, the construction in the proof of that theorem shows that $f$ is $\PSL(2,\mathbb{R})$-equivariant with respect to $\rho_0$. Finally, the choice of hyperbolic metric $g_0$ determines a Fuchsian embedding
\begin{align*}
    \iota:\pi_1(\Sigma)\to\PSL(2,\mathbb{R}).
\end{align*}
The two equivariance identities give
\begin{align*}
    \rho(\gamma)f(x)
    =f(\gamma x)
    =\rho_0(\iota(\gamma))f(x)
    \qquad\forall\,\gamma\in\pi_1(\Sigma),\
    \forall\,x\in\tilde{\Sigma}.
\end{align*}
By linear fullness, $\rho=\rho_0\circ\iota$. Hence, $\rho$ is $\rho_0$-Fuchsian, proving $(iii)$.

\begin{remark}\label{punctured}
Theorem \ref{thm:area rigidity} also holds when $\Sigma$ is a punctured Riemann surface with $\chi(\Sigma)<0$, provided that $f$ is conformal and has finite Dirichlet energy. In $(ii)$, the competitors are required to have finite Dirichlet energy, and in $(iii)$ the Fuchsian embedding is required to have finite covolume.

Indeed, for a finite-energy competitor, the metric approximation in Remark \ref{rem1} can be performed on compact subsets, leaving a fixed conformal metric unchanged near the punctures. The resulting energies converge to the area, so the energy lower bound gives the same area lower bound. In the proof of Lemma \ref{Lemma: total isotropy in the equality case}, we integrate over cusp truncations. Since $|\alpha|_{g_0}\le\frac{1}{2}$ and the lengths of the boundary horocycles tend to zero, the boundary term vanishes. Equality again gives $f^*g_H=\frac{1}{8}g_0$, which is complete, and the remainder of the proof applies.
\end{remark}




\section{Structure of minimal surfaces of constant curvature}\label{sec: differentials and representations} 
\subsection{The algebra of holomorphic and anti-holomorphic differentials}\label{subsec: differentials}
\noindent
For the sake of completeness and to ease the reading, we recall in this section the setup introduced by Bryant in  \cite[Section 1]{Bryant}. Let $(M,g)$ be one of the complete simply connected Riemannian surfaces with constant curvature, namely:
\begin{enumerate}
    \item the round 2-sphere $\mathbb{S}_{r}^2$ of radius $r$ with its standard round metric;
    \item or the Euclidean space $\mathbb{R}^2$ with its standard Euclidean metric;
    \item or the hyperbolic plane $\mathbb{H}^2$ with a metric of constant negative Gaussian curvature.
\end{enumerate}
We denote by $K$ the constant Gaussian curvature of the metric $g$. 

Let 
$$\pi:\mathscr{F}\to M$$
be the frame bundle of the oriented orthonormal frames over $(M,g)$, with respect to the metric $g$. A point $q\in\mathscr{F}$ is a triple $q=(x,e_1,e_2)$ where $x\in M$ and $(e_1,e_2)$ is an oriented orthonormal frame of $T_xM$. Notice that this is a principal $O$-bundle over $M$, with $O=\operatorname{SO}(2)$ if $K>0$, $O=\operatorname{SO}(2)$ if $K=0$ and $O=\PSO(2)$ if $K<0$.\footnote{\,We remark explicitly that $\mathfrak{o}=\mathfrak{so}(2)$ in all cases, where $\mathfrak{o}$ denotes the Lie algebra of $O$.} 
In particular, $\mathscr{F}$ is a smooth manifold which is diffeomorphic to the group $G$ of orientation preserving isometries of $M$, i.e.
\begin{enumerate}
\item $G=\operatorname{SO}(3)$ if $K>0$, 
\item $G=\mathbb{R}^2\rtimes\operatorname{SO}(2)$ if $K=0$, 
\item $G=\operatorname{PSL}(2,\mathbb{R})$ if $K<0$.
\end{enumerate}

\smallskip
\noindent
The so-called \textit{canonical 1-forms} on $\mathscr{F}$ are the unique 1-forms $\omega^1,\omega^2\in\Omega^1(\mathscr{F})$ satisfying 
\begin{align*}
    d\pi_q=\omega_{q}^1e_1+\omega_{q}^2e_2 \qquad\forall\,q=(x,e_1,e_2)\in\mathscr{F}.
\end{align*}
Let $\nabla$ be the Levi--Civita connection on $(M,g)$. Then, $\nabla$ induces a unique 1-form $\rho\in\Omega^1(\mathscr{F})$ satisfying
\begin{align}\label{Equation: 9}
\begin{cases}
    d\omega^1=-\rho\wedge\omega^2\\
    d\omega^2=\rho\wedge\omega^1,
\end{cases}
\end{align}
called the \textit{connection 1-form} associated with $\nabla$
.
Notice that $\{\omega^1,\omega^2,\rho\}$ is a coframing of $\mathscr{F}$. Hence, we have\
\begin{align*}
    d\rho=f\,\omega^1\wedge\omega^2+g\,\omega^1\wedge\rho+h\,\omega^2\wedge\rho,
\end{align*}
for some functions $f,g,h:\mathscr{F}\to\r$. By differentiating \eqref{Equation: 9} though, we have
\begin{align*}
    d\rho\wedge\omega^1=d\rho\wedge\omega^2=0,
\end{align*}
which implies $g=h\equiv 0$. Hence, we conclude that 
\begin{align*}
    d\rho=f\,\omega^1\wedge\omega^2.
\end{align*}
Notice that $f:\mathscr{F}\to\r$ is a well-defined function on $M$, because given any two points $q_1=(x,e_1,e_2),q_2=(x,e_1',e_2')\in\mathscr{F}$ we have $f(q_1)=f(q_2)$. This is precisely the Gaussian curvature of $(M,g)$ and, hence, it is constantly equal to $K$. We conclude that
\begin{align}\label{Equation: 10}
    d\rho=K\,\omega^1\wedge\omega^2.
\end{align}
For our purposes, it is convenient to use a complex coframing of $\mathscr{F}$ instead of the given one. Thus, let 
$$\omega:=\omega^1+i\omega^2$$ and consider the complex coframing of $\mathscr{F}$ given by $\{\omega,\overline{\omega},\rho\}$. Under this notation, \eqref{Equation: 9} and \eqref{Equation: 10} may be rewritten as
\begin{align}\label{Equation: 11}
\begin{cases}
    d\omega=i\,\rho\wedge\omega\vspace{1mm}\\
    \displaystyle{d\rho=\frac{iK}{2}\omega\wedge\overline{\omega}}.
\end{cases}
\end{align}
We let $L\to M$ be the complex line bundle of the 1-forms which are complex multiples of $\omega$ and $L^{-1}\to M$ be the complex line bundle of the 1-forms which are complex multiples of $\overline{\omega}$. More explicitly, we let $L,L^{-1}\subset T^*M\otimes_{\mathbb{R}}\mathbb{C}$ be the sub-bundles of the complexified cotangent bundle $T^*M\otimes_{\mathbb{R}}\mathbb{C}$ of $M$ given by
\begin{align*}
    L&:=\big\{(x,\alpha) \mbox{ : } \alpha_{\pi(q)}\circ d\pi_q=s(q)\omega_q \mbox{ for some } s(q)\in\mathbb{C},\,\forall\,q\in\pi^{-1}(x)\big\}\\
    L^{-1}&:=\big\{(x,\alpha) \mbox{ : } \alpha_{\pi(q)} \circ d\pi_q=s(q)\overline{\omega}_q \mbox{ for some } s(q)\in\mathbb{C},\,\forall\,q\in\pi^{-1}(x)\big\}
\end{align*}
For every $m\ge 0$, let $L^m\to M$ and $L^{-m}\to M$ be the $m$th powers\footnote{\,As complex line bundles, i.e. the tensor product of $L$ with itself $m$ times.} of $L$ and of $L^{-1}$respectively. By construction, for every $m\in\z$ and for every smooth section $\sigma\in C^{\infty}(L^m)$ of $L^m$, we may write $\pi^*\sigma=s_{\sigma}\,\omega^m$ for some unique function $s_{\sigma}\in C^{\infty}(\mathscr{F})$. By computing the covariant derivative of $\sigma$ with respect to the connection induced by $\rho$ on the line bundle $L^m$, we get
\begin{align}\label{Equation: 16}
    ds_{\sigma}= - m i s_{\sigma} \rho + s_{\sigma}'\omega+s_{\sigma}''\overline{\omega}
\end{align}
for some unique functions $s_{\sigma}',s_{\sigma}''\in C^{\infty}(\mathscr{F})$. Let 
\begin{align*}
    &\partial_m:C^{\infty}(L^m)\to C^{\infty}(L^{m+1})\\
    &\overline{\partial}_m:C^{\infty}(L^m)\to C^{\infty}(L^{m-1})
\end{align*}
be given by
\begin{align*}
    \pi^*(\partial_m\sigma)&=s_{\sigma}'\omega^{m+1}\\
    \pi^*(\overline{\partial}_m\sigma)&=s_{\sigma}''\omega^{m-1}.
\end{align*}
Define the $\z$-graded vector space
\begin{align*}
    \mathscr{T}:=\bigoplus_{m\in\z}C^{\infty}(L^m)
\end{align*}
and the operators $X,Y,Z$ on $\mathscr{T}$ given by
\begin{align*}
    X:=\bigoplus_{m\in\z}\partial_m \qquad Y:=\bigoplus_{m\in\z}\overline{\partial}_m \qquad Z:=\bigoplus_{m\in\z}m I_m,
\end{align*}
where $I_m$ denotes the identity on $C^{\infty}(L^m)$. The following proposition and its proof can be found in \cite[Proposition 1.1]{Bryant}.
\begin{proposition}\label{Proposition: commutation of X,Y,Z}
    The following facts hold.
    \begin{enumerate}[(i)]
        \item The operators $X,Y,Z$ satisfy
        \begin{align}\label{Equation: commuting relations X,Y,Z}
            [Z,X]=X, \qquad [Z,Y]=-Y, \qquad [X,Y]=-\frac{K}{2}Z.
        \end{align}
        In particular, they generate a Lie algebra.
        \item The operator $\Delta:=2(XY+YX)$ is the Laplace--Beltrami operator on every graded piece of $\mathscr{T}$.
        \item The operator $\Phi:=\Delta - K Z^2$ commutes with $X$, $Y$ and $Z$. 
    \end{enumerate}
\end{proposition}
We continue to follow the conventions of \cite{Bryant}.
Given any separable real Hilbert space $H$, let $H_\mathbb{C}:=H\otimes_{\mathbb{R}}\mathbb{C}$ be its complexification.
We set $\mathscr{H}:=H\otimes_{\mathbb{R}}\mathscr{T}$ and we extend the operators $X,Y,Z$ to $\mathscr{H}$ by linearity. We also have a complex bilinear pairing  $\langle\cdot,\cdot\rangle:\mathscr{H}\times\mathscr{H}\to\mathscr{T}$ which is induced by the real inner product on $H$ and the natural pairing on $\mathscr{T}$. An element $\sigma\in C^{\infty}(\mathscr{H})$ is in the $m$-th graded piece if $\pi^*\sigma=s\,\omega^m$ for some well-defined function $s\in C^{\infty}(\mathscr{F},H_{\mathbb{C}})$. 
The conjugation operation on $C^{\infty}(\mathscr{F},H_{\mathbb{C}})$ is defined in the usual way. 
The conjugation operation on $\mathscr{H}$ is then defined by the formula $\pi^*\overline{\sigma}=\overline{s}\,\omega^{-m}$ for $\pi^*\sigma=s\,\omega^m$. We remark explicitly that
\begin{align}\label{Equation: switch formulas}
    X\overline{\sigma}=\overline{Y\sigma}, \qquad Y\overline{\sigma}=\overline{X\sigma}, \qquad Z\overline{\sigma}=-\overline{Z\sigma}.
\end{align}
\noindent
From now on, define the real constants $\{A_k\}_{k\in\mathbb{N}}$ by
\begin{align}\label{Equation: 15}
    \begin{cases}
        A_0=1\\
        A_{k+1}=\displaystyle{\frac{1}{2}\bigg(1-\binom{k+1}{2}K\bigg)A_k}  \qquad\qquad\forall\,k\in\mathbb{N}
    \end{cases}
\end{align}
and the real constants $\{c_p\}_{p\in\mathbb N\smallsetminus\{0\}}$ by
\begin{align}\label{def constants c_p}
    c_p=\frac{1}{2}\bigg(1-\binom{p}{2}K\bigg)
    \qquad\forall\,p\in\mathbb N\smallsetminus\{0\}.
\end{align}

\vspace{1em}

In the remainder of this subsection, we will always make the following assumptions:

\vspace{1em}

\textbf{Assumptions:} \emph{Let $H$ be a separable real Hilbert space. Let $u:M\to{\s_H}\subset H$ be a map from the surface $(M,g)$ with constant curvature $K$ into the unit sphere ${\s_H}$ of $H$. Suppose that $u$ is both a minimal immersion and a Riemannian isometric immersion.}

\vspace{1em}

The next two propositions are from \cite[Section 1]{Bryant}.
\begin{proposition} \label{propa}
     For every $k>0$ we have 
    \begin{align*}
        YX^ku&=-c_kX^{k-1}u,\\
        XY^ku&=-c_kY^{k-1}u.
    \end{align*}
\end{proposition}
\begin{proof}
   See \cite[Proposition 1.3]{Bryant}.
\end{proof}
\begin{proposition}\label{prop: the values of R_0 and R_1}
    For every $k\ge 0$ we have 
    \begin{align*}
        \langle X^{k}u,Y^ku\rangle&=A_k\\
        \langle X^{k+1}u,Y^ku\rangle&=\langle X^{k}u,Y^{k+1}u\rangle=0.
    \end{align*}
\end{proposition}
\begin{proof}
   See \cite[Proposition 1.4]{Bryant}.
\end{proof}

The total isotropy assumption determines the remaining pairings.
\begin{proposition}\label{prop: differential of R_m}
    Suppose moreover that $u$ is totally isotropic. Then, for every $p,q\geq0$, we have
    \begin{align}\label{Equation: pairings of the differentials}
        \langle X^pu,Y^qu\rangle=A_p\delta_{pq}.
    \end{align}
\end{proposition}
\begin{proof}
    In a local complex coordinate, $X^pu$ is a linear combination of the holomorphic derivatives of $u$ of orders $1,\dots,p$. Hence, total isotropy gives
    \begin{align*}
        \langle X^pu,X^qu\rangle=0
        \qquad\forall\,p,q\geq1.
    \end{align*}
    Since $\langle u,u\rangle=1$, we also have $\langle Xu,u\rangle=0$. For every $p\geq2$,
    \begin{align*}
        X\langle X^{p-1}u,u\rangle
        &=\langle X^pu,u\rangle
        +\langle X^{p-1}u,Xu\rangle\\
        &=\langle X^pu,u\rangle.
    \end{align*}
    Thus, induction gives
    \begin{align*}
        \langle X^pu,u\rangle=0
        \qquad\forall\,p\geq1.
    \end{align*}

    \noindent
    We prove \eqref{Equation: pairings of the differentials} for $p\geq q$ by induction on $q$. The case $q=0$ follows from the previous identity and $\langle u,u\rangle=1$. Suppose that $q\geq1$ and that the statement holds with $q-1$ in place of $q$. Since $p>q-1$, Proposition \ref{propa} gives
    \begin{align*}
        0
        &=Y\langle X^pu,Y^{q-1}u\rangle\\
        &=-c_p\langle X^{p-1}u,Y^{q-1}u\rangle
        +\langle X^pu,Y^qu\rangle.
    \end{align*}
    The induction hypothesis and $c_pA_{p-1}=A_p$ prove the desired identity for $p\geq q$. The case $p<q$ follows by conjugation.
\end{proof}
\subsection{The osculating space and its properties}\label{subsec: osculating space} 

Recall that the constants $\{A_k\}$ are defined in (\ref{Equation: 15}). Similarly to \cite{Calabi,Kenmotsu}, we consider the osculating space  of the minimal isometric immersion.

\begin{definition}[The osculating space]\label{Definition: osculating space}
    Let $H$ be a separable real Hilbert space. Let $u:M\to{\s_H}\subset H$ be a smooth totally isotropic minimal isometric immersion of $M$ in the unit sphere ${\s_H}$ of $H$. 
    For every $k\in\mathbb{Z}$, let $u_k\in C^{\infty}(\mathscr{F},H_{\mathbb{C}})$ be given by
    \begin{align}\label{Equation: definition of u_k}
        \left\{ \begin{array}{rcl}
        \pi^*(X^ku)= \sqrt{A_k}u_k\omega^k & \mbox{if}
        & k\geq0 \\ 
        \pi^*(Y^{|k|}u)= \sqrt{A_{|k|}}u_k\omega^k & \mbox{if} & k<0.
        \end{array}\right.
    \end{align}
    By Proposition \ref{prop: the values of R_0 and R_1}, we have $A_k\geq0$ for every $k\in\n$, and $X^ku=Y^ku=0$ whenever $A_k=0$. In this case, we set $u_k=u_{-k}=0$.

    \noindent
    The \textit{osculating space of the map $u$ at the point $q\in\mathscr{F}$} is the closed subspace of $H_{\mathbb{C}}$ given by
    \begin{align*}
        \mathbf{O}(q,u):=\text{ closure of $\operatorname{span}_{\mathbb{C}}\{u_k(q)
        \mbox{ : } k\in\mathbb{Z}\}$ inside $H_{\mathbb{C}}$.}
    \end{align*}
    The \textit{osculating space of the map $u$} is the vector bundle over $\mathscr{F}$ given by $p:\mathbf{O}(u)\to\mathscr{F}$, where 
    \begin{align*}
        &\mathbf{O}(u):=\bigsqcup_{q\in\mathscr{F}}\mathbf{O}(q,u)\subset\mathscr{F}\times H_{\mathbb{C}}\\
        &p((q,v)):=q \qquad\forall\,(q,v)\in\mathbf{O}(u). 
    \end{align*}
\end{definition}
\begin{remark}
Since $u$ takes values in the real Hilbert space $H$, we have $\overline{u}=u$. 
    More generally, notice that, by \eqref{Equation: switch formulas}, we have 
    \begin{align*}
        \sqrt{A_k}u_{-k}\omega^{-k} = \pi^*(Y^ku)=\pi^*(\overline{X^ku})=\sqrt{A_k}\overline{u_k}\,\omega^{-k} \qquad\forall\,k\in\n.
    \end{align*}
    Hence, using also the preceding convention when $A_k=0$, for every $k\in\n$
    $$u_{-k} = \overline{u_k}.$$
    The map $u_k\in C^{\infty}(\mathscr{F},H_{\mathbb{C}})$ can be thought of as the $k$-th holomorphic differential of $u$, whilst $u_{-k}\in C^{\infty}(\mathscr{F},H_{\mathbb{C}})$ can be interpreted as the $k$-th anti-holomorphic differential of $u$.
\end{remark}
\begin{lemma}\label{Lemma: orthogonal basis for the osculating space at each point}
    Under the same assumptions and notations of Definition \ref{Definition: osculating space}. For every $q\in\mathscr{F}$ the set 
    \begin{align*}
        \{u_k(q) \mbox{ : } k\in\z,\ A_{|k|}>0\}
    \end{align*}
    forms an orthonormal Hilbert basis of the complex Hilbert space $\mathbf{O}(q,u)$.
\end{lemma}
\begin{proof}
    The family $\{u_k(q) \mbox{ : } k\in\z,\ A_{|k|}>0\}$ spans a dense subspace of $\mathbf{O}(q,u)$ by definition. Since $u_{-k}=\overline{u_k}$, Proposition \ref{prop: differential of R_m} gives
    \begin{align*}
        \langle u_k(q),\overline{u_j(q)}\rangle=\delta_{kj}
        \qquad\mbox{ if }\,kj\geq0.
    \end{align*}
    If $kj<0$, the same pairing vanishes by total isotropy and its complex conjugate. Hence, the family is orthonormal.
\end{proof}

%

For the next lemma, recall that the constants $\{c_p\}$ are defined in (\ref{def constants c_p}).
\begin{lemma}\label{Lemma: the differentials of the u_k}
    Under the same assumptions and notations of Definition \ref{Definition: osculating space}, the following recursive formulas hold for every $k\in\z$ such that $A_{|k|}>0$:
    \begin{align*}
        du_k &= -\sqrt{c_{-k}}u_{k+1}\omega+\sqrt{c_{-k+1}}u_{k-1}\overline{\omega}
        -kiu_k\rho\quad\text{ when } k\leq-2\\ 
        du_{-1} &= -\frac{1}{\sqrt{2}} u_0 \omega + \sqrt{c_2} u_{-2} \bar{\omega} + iu_{-1}\rho\\ 
        du_0 &=\frac{1}{\sqrt{2}}\big(u_{1}\omega + \overline{u_1}\,\overline{\omega}\big)\\   
        du_1 &=\sqrt{c_2}u_2 \omega - \frac{1}{\sqrt{2}} u_0 \overline{\omega} - iu_1 \rho\\
        du_k &=\sqrt{c_{k + 1}} u_{k + 1}\omega - \sqrt{c_k} u_{k - 1} \overline{\omega} - kiu_k \rho \quad \text{ when } k \geq 2.  
    \end{align*} 
\end{lemma}
\begin{proof}
    These formulas are proved verbatim in \cite[Proof of Theorem 1.6, page 264]{Bryant} by recalling that, in the notation of \cite{Bryant} for the $f_k$, we have $\sqrt{2}f_k=u_k$ for every $k>0$. Otherwise, the reader can derive them directly from \eqref{Equation: 16}. 
\end{proof}
We record the following standard consequence of elliptic regularity:
\begin{lemma}\label{Lemma: analyticity of u}
    Let $(M,g)$ be a real-analytic Riemannian surface and let $H$ be a separable Hilbert space. Let $u:M\to\mathbb{S}_H$ be a smooth solution of
    \begin{align*}
        \Delta u=-2u \qquad\mbox{ on }\,M,
    \end{align*}
    where $\Delta$ denotes the Laplace--Beltrami operator of $g$. Then $u$ is real-analytic.
\end{lemma}
\begin{proof}
    For every fixed $v\in H$, the function $f:=\langle u,v\rangle$ satisfies $\Delta f=-2f$. Since the coefficients of this elliptic equation are real-analytic, $f$ is real-analytic. Thus, $u$ is weakly-analytic in the sense of \cite[Definition 5.7]{BS71}, and \cite[Theorem 7.4]{BS71} implies that $u$ is real-analytic.
\end{proof}
%
%
\begin{proposition}\label{Proposition: the osculating space of a minimal isometric immersion is trivial}
    Under the same assumptions and notation of Definition \ref{Definition: osculating space}, there exists a linear subspace 
    $$\mathbf{T}(u)\subset H_{\mathbb{C}}$$ such that $\mathbf{O}(q,u)=\mathbf{T}(u)$ for every $q\in\mathscr{F}$. In other words, the osculating space of $u$ is a trivial vector bundle over $\mathscr{F}$, as we
    have the splitting
    \begin{align*}
        \mathbf{O}(u)=\mathscr{F}\times \mathbf{T}(u)\subset\mathscr{F}\times H_{\mathbb{C}}.
    \end{align*}
\end{proposition}
\begin{proof}
    Set
    \begin{align*}
        \mathbf{T}(u):=
        \overline{\operatorname{span}_{\mathbb{C}}
        \{u(x)\mbox{ : }x\in M\}}\subset H_{\mathbb{C}}.
    \end{align*}
    Since derivatives of $u$ take values in $\mathbf{T}(u)$, we have $\mathbf{O}(q,u)\subset\mathbf{T}(u)$ for every $q\in\mathscr{F}$.

    \noindent
    Conversely, fix $q\in\mathscr{F}$ and $v\in\mathbf{O}(q,u)^\perp$. By Lemma \ref{Lemma: analyticity of u}, the function
    \begin{align*}
        f(x):=\langle u(x),\overline{v}\rangle
    \end{align*}
    is real-analytic. It satisfies $\Delta f=-2f$, and all its holomorphic and anti-holomorphic covariant derivatives vanish at $\pi(q)$. Choose analytic conformal coordinates $z$ near $\pi(q)$ and write $g=e^{2a}|dz|^2$. The corresponding pure derivatives of $f$ vanish at $\pi(q)$, and
    \begin{align*}
        \partial_z\partial_{\bar z}f=-\frac{e^{2a}}{2}f.
    \end{align*}
    Differentiating this identity shows inductively that every mixed derivative of $f$ also vanishes at $\pi(q)$. Therefore, analytic continuation gives $f=0$ on $M$. Hence, $\mathbf{T}(u)\subset\mathbf{O}(q,u)$, proving the statement.
\end{proof}

Proposition \ref{Proposition: the osculating space of a minimal isometric immersion is trivial} means that  
$$u(M)=u_0(\mathscr{F})\subset \mathbf{T}(u)$$
where $\mathbf{T}(u)$ is a subspace of the Hilbert space $H_{\mathbb C}$ depending only on the map $u$.
\section{Classification of constant negative curvature minimal surfaces}\label{sec: classification} 
We can now classify totally isotropic minimal surfaces in Hilbert spheres, which have constant negative curvature. Let the surface $(M,g)$, the constant Gaussian curvature $K$ and the Lie group $G$ be defined as in the beginning of Subsection \ref{subsec: differentials}. Let $\mathscr{O}_{s,v}$ be the orbit defined in (\ref{orbit}). For a map $u$ satisfying the assumptions of the previous section, we will use the notation $ \mathbf{T}(u)$ from Proposition \ref{Proposition: the osculating space of a minimal isometric immersion is trivial}.
\begin{theorem}\label{Theorem: classification negative curvature}
     Let $H$ be a separable infinite dimensional real Hilbert space. Suppose that $K<0$ and let $u:(M,g)\to{\s_H}\subset H$ be a totally isotropic minimal isometric immersion of $M$ in the unit sphere ${\s_H}$ of $H$. Then, $u(M)$ is an  orbit of a unitary irreducible representation of class one of $G=\PSL(2,\r)$ on $\mathbf{T}(u)$. More precisely, $u(M)$ is equal to $\mathscr{O}_{s,v}$ for some $s\in i\r\cup (-\frac{1}{2},\frac{1}{2})$ such that $K=-\frac{8}{1-4s^2}$, and some $\PSO(2)$-invariant unit vector $v\in \mathbf{T}(u)$.
\end{theorem}
\begin{proof}
    Let $q\in\mathscr{F}$ be any point in the frame bundle $\mathscr{F}$. Let $\mathbf{T}(u)$ be the linear subspace of $H_\mathbb{C}$ given by Proposition \ref{Proposition: the osculating space of a minimal isometric immersion is trivial}, which we recall is the common osculating space of $u$ at any point of $\mathscr{F}$.
    Recall that the constants $\{A_k\}$ and $\{c_p\}$ are defined in (\ref{Equation: 15}) and (\ref{def constants c_p}), and depend on the Gaussian curvature $K$. We will use standard notions from functional analysis, whose definitions are recalled in Appendix \ref{Notions from functional analysis}. Define 
    \begin{align*}
        e_k:= u_k(q) \in \mathbf{T}(u) \qquad&\forall\,k\in\mathbb{Z}.
    \end{align*}
    By Lemma \ref{Lemma: orthogonal basis for the osculating space at each point}, $\{e_k \mbox{ : } k\in\z\}$ is an orthonormal Hilbert basis of $\mathbf{T}(u)$. Let $D\subset \mathbf{T}(u)$ be the dense linear subspace of $\mathbf{T}(u)$ given by 
    \begin{align*}
        D:=\operatorname{span}_{\mathbb{C}}\{e_k \mbox{ : } k\in\mathbb{Z}\}.
    \end{align*}

    \noindent
    Define the linear operators $T_X,T_Y,T_Z\in\mathcal{L}(D)$ given by
    \begin{align*}
        T_Xe_k&:=\begin{cases}
            -\sqrt{c_{-k}}e_{k+1}   & \mbox{ if } k\le -2\\
            -\frac{1}{\sqrt{2}}e_0         & \mbox{ if } k=-1\\
            \frac{1}{\sqrt{2}}e_1                     & \mbox{ if } k=0\\
            \sqrt{c_{k+1}}e_{k+1}   & \mbox{ if } k\ge 1\\
        \end{cases}\\
        T_Ye_k&:=\begin{cases}
            \sqrt{c_{-k+1}}e_{k-1}   & \mbox{ if } k\le -1\\
            \frac{1}{\sqrt{2}}e_{-1}                  & \mbox{ if } k=0\\
            -\frac{1}{\sqrt{2}}e_0         & \mbox{ if } k=1\\
            -\sqrt{c_k}e_{k-1}      & \mbox{ if } k\ge 2\\
        \end{cases}\\
        T_Ze_k&:=ke_k \qquad\forall\, k\in\mathbb{Z},
    \end{align*}
    and extended by $\mathbb{C}$-linearity. 

    \begin{remark}
        The operators $T_X,T_Y,T_Z$ record the differential formulas of Lemma \ref{Lemma: the differentials of the u_k} at the fixed frame $q$. More precisely,
        \begin{align*}
            (du_k)_q
            =(T_Xe_k)\omega_q+(T_Ye_k)\overline{\omega}_q
            -i(T_Ze_k)\rho_q
            \qquad\forall\,k\in\z.
        \end{align*}
        Thus, $T_Xe_k$, $T_Ye_k$ and $-iT_Ze_k$ are respectively the coefficients of the holomorphic, anti-holomorphic and vertical components of $(du_k)_q$.
    \end{remark}
    
    We claim that
    \begin{align}\label{Equation: adjoints of X,Y,Z}
    \begin{split}
        \langle T_Xv,\overline{w}\rangle&=-\langle v,\overline{T_Yw}\rangle \qquad\forall v,w\in D\\
        \langle T_Yv,\overline{w}\rangle&=-\langle v,\overline{T_Xw}\rangle \qquad\forall v,w\in D\\
        \langle T_Zv,\overline{w}\rangle&=\langle v,\overline{T_Zw}\rangle \hspace{-0.45mm}\quad\qquad\forall v,w\in D.
    \end{split}
    \end{align}
    It is straightforward to see that $\langle T_Zv,\overline{w}\rangle=\langle v,\overline{T_Zw}\rangle$ for every $v,w\in D$. Thus, we just focus on the other two relations. 
   Since for any $k$, $T_Xe_k$ is a complex multiple of $e_{k+1}$ and $T_Ye_k$ is a complex multiple of $e_{k-1}$, and since $\{e_k\}$ forms an orthonormal basis of $\mathbf{T}(u)$ by Lemma \ref{Lemma: orthogonal basis for the osculating space at each point}, we readily obtain from the definitions of $T_X,T_Y$ that 
    \begin{align*}
        \langle T_Xe_k,\overline{e_j}\rangle&:=\begin{cases}
            0   & \mbox{ if } j\neq k+1\\
              -\sqrt{c_{-k}}  & \mbox{ if } j=k+1 \leq -1 \\
     -\frac{1}{\sqrt{2}}  & \mbox{ if } j=k+1 = 0\\
            \frac{1}{\sqrt{2}}        & \mbox{ if } j=k+1= 1\\
            \sqrt{c_{k+1}}         & \mbox{ if } j=k+1\geq 2\\
        \end{cases}
           \end{align*}
    
  \begin{align*}
           \langle e_k,\overline{T_Y e_j}\rangle&:=\begin{cases}
            0   & \mbox{ if } j\neq k+1\\
              \sqrt{c_{-k}}  & \mbox{ if } j=k+1 \leq -1 \\
     \frac{1}{\sqrt{2}}  & \mbox{ if } j=k+1 = 0\\
            -\frac{1}{\sqrt{2}}        & \mbox{ if } j=k+1= 1\\
            -\sqrt{c_{k+1}}         & \mbox{ if } j=k+1\geq 2.\\
        \end{cases}
    \end{align*}
    Thus, we conclude that
    \begin{align*}
        \langle T_Xv,\overline{w}\rangle&=-\langle v,\overline{T_Yw}\rangle \qquad\forall v,w\in D.
    \end{align*}
    By taking the conjugate on both sides of the previous equality, we get 
    \begin{align*}
        \langle T_Yv,\overline{w}\rangle&=-\langle v,\overline{T_Xw}\rangle \qquad\forall v,w\in D
    \end{align*}
    and our claim (\ref{Equation: adjoints of X,Y,Z}) follows. 
    Set 
    $$F_{1} := T_{X} + T_{Y},$$
    $$F_2 := i(T_X - T_Y),$$
    $$F_3 :=-iT_Z.$$
    By \eqref{Equation: adjoints of X,Y,Z}, we have
    \begin{align*}
        \langle F_1v,\overline{w}\rangle&=-\langle v,\overline{F_1w}\rangle \qquad\forall v,w\in D\\
        \langle F_2v,\overline{w}\rangle&=-\langle v,\overline{F_2w}\rangle \qquad\forall v,w\in D\\
        \langle F_3v,\overline{w}\rangle&=-\langle v,\overline{F_3w}\rangle \qquad\forall v,w\in D. 
    \end{align*}

    \noindent
Thus, $F_1,F_2,F_3$ are skew-symmetric on the common invariant dense domain $D$. We claim that they are essentially skew-adjoint in the sense of Definition \ref{Definition: self and skew-adjoint operators}. When exponentiating these operators, we use their skew-adjoint closures.

\noindent
As the proofs for $F_1$ and $F_2$ are similar, let us write the details for $F_2=i(T_X-T_Y)$ for instance. It suffices to show that $T_X - T_Y$ is essentially self-adjoint, so it is enough to show that $T_X - T_Y$ is a Jacobi matrix operator (see Definition \ref{Definition: Jacobi matrix operator}) satisfying the assumptions of Carleman's test (Lemma \ref{Lemma: Carleman's test}). By definition of $T_X,T_Y$, the operator $T_X - T_Y$ can be represented by an infinite symmetric tridiagonal matrix in the orthonormal basis $\{e_k\}_{k\in \mathbb{Z}}$. Notice that in this basis, $T_X - T_Y$ has two coefficients approaching positive and negative infinity: 
\begin{align*}
    \langle (T_X-T_Y)e_{k-1},\overline{e_k}\rangle
    &=\sqrt{c_k}\sim\frac{\sqrt{|K|}}{2}k
    \qquad&\mbox{ as }\,k\to+\infty;\\
    \langle (T_X-T_Y)e_{k-1},\overline{e_k}\rangle
    &=-\sqrt{c_{1-k}}\sim-\frac{\sqrt{|K|}}{2}|k|
    \qquad&\mbox{ as }\,k\to-\infty.
\end{align*}
After the change of basis $e_{1-2k}\mapsto-e_{1-2k}$ for every $k\geq1$, the off-diagonal coefficients are positive real numbers $b_k$ satisfying
\begin{align*}
    b_k\sim\frac{\sqrt{|K|}}{2}|k|
    \qquad\mbox{ as }\,|k|\to+\infty.
\end{align*}
Hence,
\begin{align*}
    \sum_{k\geq0}\frac{1}{b_k}=\infty
    \qquad\mbox{ and }\qquad
    \sum_{k\leq0}\frac{1}{b_k}=\infty.
\end{align*}
Carleman's test therefore implies that $T_X-T_Y$ is essentially self-adjoint, and $F_2$ is essentially skew-adjoint. The case of $F_1$ is similar.

\noindent
We now claim that $F_3=-iT_Z$ is essentially skew-adjoint, namely that $T_Z$ is essentially self-adjoint. This is a consequence of the corollary of \cite[Theorem VIII.3]{ReedSimon}: since $T_Z$ is symmetric on its domain, and since the range of $T_Z\pm i$ is clearly dense in the Hilbert space $\mathbf{T}(u)$, we deduce that $T_Z$ is essentially self-adjoint as wanted.

\noindent
Next, we show that $\{F_1, F_2, F_3\}$ spans the Lie algebra of $G$. Recall that, by Proposition \ref{Proposition: commutation of X,Y,Z}, we have 
$$[T_Z, T_X] = T_X,$$ 
$$[T_{Z}, T_{Y}] = -T_{Y},$$
$$[T_X, T_Y] = -\frac{K}{2}T_Z.$$ Thus, we obtain the Lie bracket of $F_1, F_2$, and $F_3$ as follows:
    \begin{align*}
        [F_1, F_2] &= i[T_X + T_Y, T_X - T_Y] = -2i[T_X, T_Y] = |K|F_3;\\ 
        [F_2, F_3] &= [T_X - T_Y,T_Z] = -(T_X + T_Y) = -F_1;\\ 
        [F_3, F_1] &= -i[T_Z, T_X + T_Y] = -i(T_X - T_Y) = -F_2.
    \end{align*}
    Recall the definition of the basis $\{\sigma_1,\sigma_2,\sigma_3\}\subset\mathfrak{sl}(2,\mathbb{R})$ of $\mathfrak{sl}(2,\mathbb{R})$ from Lemma \ref{Lemma: a useful basis of sl(2,R)}. The map $\varphi:\mathfrak{sl}(2,\mathbb{R})\to\mathfrak{u}(\mathbf{T}(u))$ given by
    \begin{align}
        \varphi(\sigma_1):=\frac{1}{\sqrt{|K|}}F_1 \qquad \varphi(\sigma_2):=\frac{1}{\sqrt{|K|}}F_2 \qquad \varphi(\sigma_3):=F_3
    \end{align}
    and extended by $\mathbb{R}$-linearity defines a Lie algebra representation of $\mathfrak{sl}(2,\mathbb{R})$ on $\mathbf{T}(u)$. Here linear combinations and commutators are understood on $D$, and the resulting operators are identified with their closures.
    
    \noindent
    \textbf{Claim 1}. The Lie algebra representation $\varphi$ is irreducible. In order to show this, it is enough to check that given a $\mathfrak{sl}(2,\mathbb{R})$-invariant linear subspace $W\subset D:=\operatorname{span}_{\mathbb{C}}\{e_k \mbox{ : } k\in\mathbb{Z}\}$ such that $W\neq\{0\}$, we must have $W=D$. Since $W\neq\{0\}$, there exists $p,q\in\mathbb{N}$ and $\alpha_1,\dots,\alpha_q,\beta_1,\dots\beta_p\in\mathbb{C}$ and $i_1<\dots<i_q< 0\le j_1<\dots<j_p\in\mathbb{Z}$ such that
    \begin{align*}
        v:=\alpha_1e_{i_1}+\dots+\alpha_qe_{i_q}+\beta_1e_{j_1}+\dots+\beta_pe_{j_p}\in W.
    \end{align*}
    Assume $\beta_p\neq0$. Since $W$ is $\mathfrak{sl}(2,\mathbb{R})$-invariant and using that $T_Ze_0=0$, we have that
    \begin{align*}
        \Big(T_{X}^{j_{p-1}}T_ZT_{Y}^{j_{p-1}}\dots T_X^{j_1}T_ZT_Y^{j_1}T_{Y}^{|i_q|}T_ZT_{X}^{|i_q|}\dots T_Y^{|i_1|}T_ZT_X^{|i_1|}\Big)v=\lambda e_{j_p}\in W,
    \end{align*}
    for some nonzero $\lambda\in\mathbb{C}$ depending on $\alpha_1,\dots,\alpha_q,\beta_1,\dots\beta_p$ and $i_1<\dots<i_q\le 0\le j_1<\dots <j_p$. Since $W$ is a linear subspace, we have that $e_{j_p}\in W$. But then, by repeatedly applying either $T_X$ or $T_Y$ to $e_{j_p}\in W$ and using the fact that $W$ is a linear subspace we obtain that $e_k\in W$ for every $k\in\mathbb{Z}$.
    
    If instead we have $\beta_1=\dots=\beta_p=0$ but $\alpha_q\neq0$, again since $W$ is $\mathfrak{sl}(2,\mathbb{R})$-invariant, we have that
    \begin{align*}
        \Big(T_{Y}^{|i_{q-1}|}T_ZT_{X}^{|i_{q-1}|}\dots T_Y^{|i_1|}T_ZT_X^{|i_1|}\Big)v=\lambda e_{i_q}\in W,
    \end{align*}
    for some nonzero $\lambda\in\mathbb{C}$ depending on $\alpha_1,\dots,\alpha_q,\beta_1,\dots\beta_p$ and $i_1<\dots<i_q\le 0$. Since $W$ is a linear subspace, we have that $e_{i_q}\in W$. But then, by repeatedly applying either $T_X$ or $T_Y$ to $e_{i_q}\in W$, we obtain that $e_k\in W$ for every $k\in\mathbb{Z}$. Thus, our Claim 1 follows.
    
    \medskip
    \noindent
    Next, for every $A\in G$, let $\rho(A)$ be the unique unitary operator on $\mathbf{T}(u)$ satisfying
    \begin{align*}
        \rho(A)e_k=u_k(Aq)
        \qquad\forall\,k\in\z.
    \end{align*}
    Such an operator exists by Lemma \ref{Lemma: orthogonal basis for the osculating space at each point}. Set $x_0:=\pi(q)$. Since the operators $X,Y$ commute with pullback by orientation preserving isometries, the maps $\rho(A)u$ and $u\circ A$ have the same value and the same holomorphic and anti-holomorphic covariant derivatives of every order at $x_0$. Both maps satisfy $\Delta f=-2f$. As in the proof of Proposition \ref{Proposition: the osculating space of a minimal isometric immersion is trivial}, this equation implies that their mixed derivatives also agree at $x_0$. By real analyticity and connectedness of $M$, we obtain
    \begin{align*}
        \rho(A)u(x)=u(Ax)
        \qquad\forall\,A\in G,\quad\forall\,x\in M.
    \end{align*}
    Consequently,
    \begin{align*}
        \rho(A)\rho(B)u(x)
        =u(ABx)
        =\rho(AB)u(x)
        \qquad\forall\,A,B\in G,\quad\forall\,x\in M.
    \end{align*}
    Since the closed complex linear span of $u(M)$ is $\mathbf{T}(u)$, we conclude that $\rho(A)\rho(B)=\rho(AB)$. Moreover, $A\mapsto\rho(A)e_k=u_k(Aq)$ is continuous for every $k\in\z$. By density of $D$ and unitarity, $\rho$ is strongly continuous.
    
    \noindent
    Under the identification $G\to\mathscr{F}$ given by $A\mapsto Aq$, consider the infinitesimal generators corresponding to the tangent vectors dual to $\omega^1,\omega^2,\rho$ at $q$. Differentiating $\rho(A)e_k=u_k(Aq)$ and applying Lemma \ref{Lemma: the differentials of the u_k}, we see that these generators restrict to $F_1,F_2,F_3$ on $D$, respectively. Since these operators are essentially skew-adjoint, the generators are their skew-adjoint closures.
    
    \noindent
    To prove irreducibility, let $W\subset\mathbf{T}(u)$ be a closed $\rho(G)$-invariant subspace, and let $P$ be the orthogonal projection onto $W$. By unitarity, $P$ commutes with $\rho(G)$ and hence with its infinitesimal generators. Since
    \begin{align*}
        F_3e_k=-ike_k
        \qquad\forall\,k\in\z,
    \end{align*}
    and these eigenvalues are simple, we have $Pe_k\in\c e_k$ for every $k\in\z$. Thus, $P(D)\subset D$, and commutation with the generators implies that $P(D)$ is invariant under $F_1,F_2,F_3$. By Claim 1, either $P(D)=\{0\}$ or $P(D)=D$. Since $D$ is dense, we conclude that either $W=\{0\}$ or $W=\mathbf{T}(u)$. Hence, $\rho$ is irreducible.
    
    \noindent
    Finally, equivariance shows that $e_0=u(x_0)$ is fixed by the stabilizer of $x_0$ in $G$. Since this stabilizer is conjugate to $\PSO(2)$, the representation $\rho$ is of class one. By transitivity of the action of $G$ on $M$,
    \begin{align*}
        u(M)=\{\rho(A)e_0 \mbox{ : } A\in G\}.
    \end{align*}
    The statement follows from Theorem \ref{thm: irreducible representations} and Theorem \ref{thm: special orbit is minimal}.
\end{proof}
%

%
\begin{remark}
    We notice that the statement and proof of Theorem \ref{Theorem: classification negative curvature} in the case $K>0$ correspond to \cite[Theorem 1.6]{Bryant}. 
\end{remark}
\section{Minimality of 2-dimensional orbits}\label{sec: minimality of special orbits} 
In \cite{Calabi, Bryant}, a step is to check that the 2-dimensional orbits of the irreducible representation of $\operatorname{SO}(3)$ into $\operatorname{SO}(2m - 1)$ are indeed minimal spheres inside $\mathbb{S}^{2m-2}$, called Boruvka spheres.
In a similar vein, the goal of this section is to prove that the 2-dimensional orbits of any irreducible unitary representation of $\PSL(2, \mathbb{R})$ of class one are minimal surfaces. Below, $H$ is a complex Hilbert space and $\mathbb{S}_H$ denotes its unit sphere. Recall that the orbit $\mathscr{O}_{s,v}$ is defined in (\ref{orbit}). Let $\rho_s:\PSL(2, \mathbb{R})\to\End(H)$ be the irreducible unitary representation of $\PSL(2, \mathbb{R})$ given by Theorem \ref{thm: irreducible representations}.

\begin{theorem}\label{thm: special orbit is minimal}
    Let $s\in i\r\cup(-\frac{1}{2},\frac{1}{2})$. For every $\PSO(2)$-invariant unit vector $v\in\mathbb{S}_H$, the orbit $\mathscr{O}_{s,v}$ is a complete totally isotropic minimal surface in $\mathbb{S}_H$ of constant Gaussian curvature
    \begin{align*}
        K_s=-\frac{8}{1-4s^2}.
    \end{align*}
\end{theorem}
\begin{proof}
    Recall that 
    \begin{align*}
        \mathbb{H}^2=\PSL(2,\mathbb{R})/\PSO(2)
    \end{align*}
    and that the hyperbolic metric $g_0$ on $\mathbb{H}^2$ is, up to a positive, constant multiplicative factor, the unique $\PSL(2,\mathbb{R})$-invariant metric on $\mathbb{H}^2$. Inspecting the definitions of $\rho_s$ \cite{Lang} \cite[Chapter 5]{Lubotzky}, one checks that an orbit of $\rho_s$ is 2-dimensional if and only if it is equal to $\mathscr{O}_{s,v}$, where $v$ is a $\operatorname{PSO}(2)$-invariant vector.

    Now fix a $\operatorname{PSO}(2)$-invariant vector $v$ of unit norm.
    Consider the map $u:\mathbb{H}^2\to{\s_H}$ given by
    \begin{align*}
        u([A]):=\rho_s(A)v \qquad\forall\, A\in\PSL(2,\mathbb{R}).
    \end{align*}
    By $\PSO(2)$-invariance of $v$ and strong continuity of $\rho_s$, $u$ is a well-defined continuous map on $\mathbb{H}^2$. Moreover,
    \begin{align*}
        u(\mathbb{H}^2)=\mathscr{O}_{s,v}
    \end{align*}
    and we have the following natural equivariance property:
    \begin{align}\label{homog}
        u(A\cdot[B])=\rho_s(A)u([B]) \qquad\forall\,A,B\in\PSL(2,\mathbb{R}),
    \end{align}
    i.e. $u$ is $\PSL(2,\mathbb{R})$-equivariant. 

    \noindent
    By \cite[Proposition 5.1.6]{Lubotzky}, the map $\varphi:\PSL(2,\r)\to\c$ given by
    \begin{align*}
        \varphi(A):=\langle\rho_s(A)v,v\rangle \qquad\forall\,A\in\PSL(2,\r),
    \end{align*}
    is a spherical function associated with $\rho_s$. As such, via the discussion in \cite[Section  5.2, pages 65-66]{Lubotzky}, $\varphi$ is a solution on $\mathbb{H}^2$ of
    \begin{align}\label{Equation: laplacian on varphi}
        \Delta_{g_0}\varphi=-\lambda_s\varphi,
    \end{align}
    where $\lambda_s$ is defined as in Theorem \ref{thm: irreducible representations} as
    \begin{align*}
        \lambda_s:=\frac{1}{4}-s^2.
    \end{align*}
    Moreover, for every $A,B\in\PSL(2,\r)$, we have
    \begin{align*}
        \langle\rho_s(A)v,\rho_s(B)v\rangle=\varphi(B^{-1}A).
    \end{align*}
    The smoothness of $\varphi$ implies that $A\mapsto\rho_s(A)v$ is smooth. Indeed, the inner products of its difference quotients of any order are the corresponding mixed difference quotients of $(A,B)\mapsto\varphi(B^{-1}A)$, so these difference quotients are Cauchy in $H$. Hence, $u$ is smooth.
    As
    \begin{align*}
        \varphi(A)=\langle \rho_s(A)v,v\rangle=\langle u([A]),v\rangle \qquad\forall\,[A]\in\mathbb{H}^2,
    \end{align*}
    this discussion implies that the ``coordinate function'' $\langle u([A]),v\rangle$ in the direction of $v$ is an eigenfunction of the Laplacian on $\mathbb{H}^2$ with eigenvalue $\lambda_s=\frac{1}{4}-s^2$. 
    By equivariance (\ref{homog}), for any $v'\in u(\mathbb{H}^2)$, the coordinate function
    $\langle u([A]),v'\rangle$
    satisfies the same property. Since $\rho_s$ is irreducible, the closed linear span of $u(\mathbb{H}^2)$ is $H$. Hence,
    \begin{equation} \label{vector laplace}
        \Delta_{g_0} u = -\lambda_s u.
    \end{equation}
    Let $g_H$ be the standard round metric induced by $H$ on ${\s_H}$. Notice that by (\ref{homog}), the pull-back metric $u^*g_H$ is $\PSL(2,\mathbb{R})$-invariant on $\mathbb{H}^2$, which implies that 
    \begin{align*}
        u^*g_H=c_sg_0
    \end{align*}
    for some $c_s\geq0$. If $c_s=0$, then $u$ is constant and $v$ is fixed by all of $\PSL(2,\mathbb{R})$, contradicting nontrivial irreducibility. Hence, $c_s>0$. Thus $u:(\mathbb{H}^2,u^*g_H)\to({\s_H},g_H)$ is an isometric immersion such that 
    \begin{align*}
        \Delta_{u^*g_H}u=\Delta_{c_sg_0}u=-\frac{\lambda_s}{c_s}u.
    \end{align*}
    This implies that $u$ must be minimal and 
    \begin{align*}
        \frac{\lambda_s}{c_s} = 2,\quad c_s=\frac{\lambda_s}{2}=\frac{1}{8}-\frac{s^2}{2}.
    \end{align*}
    Since $K_{c_sg_0}=c_s^{-1}K_{g_0}=-c_s^{-1}$, we have
    \begin{align*}
        K_{u^*g_H}=K_{c_sg_0}=-\frac{1}{c_s}=-\frac{8}{1-4s^2}<0.
    \end{align*}
    Since $u^*g_H=c_sg_0$ and $c_s>0$, the induced metric is complete.

    \noindent
    It remains to prove total isotropy. Regard the ambient Hilbert space as a real Hilbert space and use the complex bilinear extension of its real inner product. Fix an oriented orthonormal frame $q$ at $x\in\mathbb{H}^2$, and let $A_\theta$ be the rotation through angle $\theta$ about $x$. By \eqref{homog}, the coefficient of $X^pu$ at $q$ transforms under $\rho_s(A_\theta)$ by multiplication by $e^{ip\theta}$. Invariance of the bilinear pairing therefore gives
    \begin{align*}
        \langle X^pu,X^qu\rangle
        =e^{i(p+q)\theta}\langle X^pu,X^qu\rangle
        \qquad\forall\,\theta\in\r.
    \end{align*}
    Thus, $\langle X^pu,X^qu\rangle=0$ for every $p,q\geq1$. Since the ordinary holomorphic derivatives are linear combinations of the successive holomorphic covariant derivatives, $u$ is totally isotropic. The statement follows.
\end{proof}
\begin{proof}[Proof of Theorem \ref{thm:classification}]
    Let the immersed surface be complete, connected and totally isotropic, with constant Gaussian curvature $K<0$. Its universal cover is the complete simply connected surface of curvature $K$. The lifted immersion therefore satisfies the assumptions of Theorem \ref{Theorem: classification negative curvature}, which identifies its image with $\mathscr{O}_{s,v}$. The converse follows from Theorem \ref{thm: special orbit is minimal}.
\end{proof}
\begin{example}[Necessity of total isotropy]\label{Example: necessity of total isotropy}
The total isotropy assumption in Theorem \ref{thm:classification} cannot be omitted. Let $\mu=d\theta/(2\pi)$ be the normalized arc length measure on $\s^1$, and let $H=L^2(\s^1,\mu\,;\r)$ with its usual inner product. Define the integral operator $T:H\to H$ by
\begin{align*}
    (Ta)(\xi):=\frac{1}{16}\int_{\s^1}
    \Real((\xi\eta)^2)|\xi-\eta|^3 a(\eta)\,d\mu(\eta).
\end{align*}
Since its kernel is symmetric and bounded in absolute value by $1/2$, $T$ is self-adjoint and $\|T\|\leq1/2$. Thus, we may set $S:=(\operatorname{Id}_H+T)^{1/2}$ and define $f:\D\to H$ by
\begin{align*}
    f(z):=S\mathscr{P}(z) \qquad\forall\,z\in\D,
\end{align*}
where $\mathscr{P}$ is given by \eqref{Equation: Poisson map}.
We show that $f$ is a complete minimal immersion of constant Gaussian curvature $-8$ into $\s_H$ which is not totally isotropic.

\noindent
To compute the norm of $f$ and its induced metric on $\mathbb{D}$, fix any arbitrary $z=x+iy\in\D$ and consider the M\"obius transformation
\begin{align*}
    A_z:\overline{\D}&\to\overline{\D}\\
    w&\mapsto\frac{z+w}{1+\overline{z}w}.
\end{align*}
Since $A_z$ restricts to a diffeomorphism of $\s^1$, we change variables in the double integral by setting
$\xi=A_z(\zeta)$ and $\eta=A_z(\tau)$, with $\zeta,\tau\in\s^1$.
Using $\mathscr{P}(z)(A_z(\zeta))=|A_z'(\zeta)|^{-1/2}$, we obtain
\begin{align*}
    16\langle T\mathscr{P}(z),\mathscr{P}(z)\rangle_H
    =(1-|z|^2)^4\Real\int_{\s^1}\int_{\s^1}
    \frac{\zeta^2\tau^2|\zeta-\tau|^3}
    {(1+\overline{z}\zeta)^4(1+\overline{z}\tau)^4}
    \,d\mu(\zeta)\,d\mu(\tau).
\end{align*}
Expanding the two factors in its denominator into power series, we can see that the integrand is a uniformly convergent sum of terms proportional to $\zeta^{2+j}\tau^{2+k}|\zeta-\tau|^3$, with $j,k\geq0$. Each term has zero integral by simultaneous rotation of $\zeta,\tau$, since $4+j+k>0$. Hence, $\langle T\mathscr{P}(z),\mathscr{P}(z)\rangle_H=0$ and since $z$ was arbitrary
\begin{align}\label{Equation: vanishing 1}
    \langle T\mathscr{P},\mathscr{P}\rangle_H=0 \qquad\mbox{ on } \mathbb{D}.
\end{align}
For every $\zeta\in\s^1$, direct differentiation gives
\begin{align*}
    \frac{\mathscr{P}_x(z)(A_z(\zeta))}{\mathscr{P}(z)(A_z(\zeta))}
    =\frac{\Real\zeta}{1-|z|^2},
    \qquad
    \frac{\mathscr{P}_y(z)(A_z(\zeta))}{\mathscr{P}(z)(A_z(\zeta))}
    =\frac{\Imag\zeta}{1-|z|^2}.
\end{align*}
Replacing both factors $\mathscr{P}(z)$ in the preceding integral by first derivatives can therefore decrease the sum of the exponents by at most two. This sum remains positive, so the same argument gives
\begin{align}\label{Equation: vanishing 2}
    \langle T\mathscr{P}_x,\mathscr{P}_x\rangle_H
    =\langle T\mathscr{P}_x,\mathscr{P}_y\rangle_H
    =\langle T\mathscr{P}_y,\mathscr{P}_y\rangle_H=0 \qquad\mbox{ on }\mathbb D.
\end{align}
Since $S$ is self-adjoint and $S^2=\operatorname{Id}_H+T$, we have
\begin{align}\label{Equation: useful}
    \langle Sa,Sb\rangle_H
    =\langle a,b\rangle_H+\langle Ta,b\rangle_H
    \qquad\forall\,a,b\in H.
\end{align}
Since $f=S\mathscr{P}$ and $df=S\,d\mathscr{P}$, \eqref{Equation: useful} together with \eqref{Equation: vanishing 1} and \eqref{Equation: vanishing 2} gives
\begin{align*}
    \|f(z)\|_H^2&=\|\mathscr{P}(z)\|_H^2=1 \qquad\forall\,z\in\D,\\ f^*g_H&=\mathscr{P}^*g_H=\frac{1}{8}g_0,
\end{align*}
where $g_0:=4|dz|^2/(1-|z|^2)^2$ and the last equality follows from \cite[Lemma 1.14(2)]{Song25}. In particular, the induced metric is complete and has Gaussian curvature $-8$. The identities for the Poisson kernel in the proof of Lemma \ref{Lemma: total isotropy in the equality case} give $\Delta_{g_0}\mathscr{P}=-\mathscr{P}/4$. Hence,
\begin{align*}
    \Delta_{f^*g_H}f=8S\Delta_{g_0}\mathscr{P}=-2f,
\end{align*}
which implies that $f$ is a minimal immersion in $\s_H$.

\noindent
Finally, note that $\mathscr{P}_{zz}(0)(\xi)=\frac{3}{4}\xi^{-2}$. Extending the inner product of $H$ complex bilinearly, we have 
\begin{align}\label{Equation: vanishing 3}
    \langle \mathscr{P}_{zz}(0),\mathscr{P}_{zz}(0)\rangle_H=\frac{9}{16}\int_{\s^1}\xi^{-4}\,d\mu(\xi)=0.
\end{align}
Moreover, 
\begin{align*}
    2\Real((\xi\eta)^2)(\xi\eta)^{-2}=1+(\xi\eta)^{-4} \qquad\forall\,\xi,\eta\in\mathbb S^1.
\end{align*}
Setting $\eta=\xi\zeta$ and recalling that $\int_{\s^1}\xi^{-8}\,d\mu(\xi)=0$, by \eqref{Equation: useful} and \eqref{Equation: vanishing 3} we obtain
\begin{align*}
    \langle f_{zz}(0),f_{zz}(0)\rangle_H
    &=\langle T\mathscr{P}_{zz}(0),\mathscr{P}_{zz}(0)\rangle_H\\
    &=\frac{9}{512}\int_{\s^1}\int_{\s^1}
    \big(1+(\xi\eta)^{-4}\big)|\xi-\eta|^3
    \,d\mu(\eta)\,d\mu(\xi)\\
    &=\frac{9}{512}\int_{\s^1}|1-\zeta|^3
    \bigg(1+\zeta^{-4}\int_{\s^1}\xi^{-8}\,d\mu(\xi)\bigg)
    \,d\mu(\zeta)\\
    &=\frac{9}{512}\int_0^{2\pi}
    8\sin^3\bigg(\frac{\theta}{2}\bigg)\,d\mu(\theta)
    =\frac{3}{16\pi}\neq0.
\end{align*}
Thus, $f$ is not totally isotropic. By Theorem \ref{thm: special orbit is minimal}, its image cannot be carried onto any $\mathscr{O}_{s,v}$ by a linear isometry.
\end{example}



%
\vspace{2em}
\appendix
\section{Notions from functional analysis} \label{Notions from functional analysis}
Given a separable complex inner product space $X$, we denote by $\mathcal{L}(X)$ the space of the linear operators on $X$ and by $\mathcal{B}(X)$ the space of the bounded linear operators on $X$. For the sake of completeness, we recall here some useful standard definitions and lemmas in linear functional analysis. Our references are \cite[Chapter VIII]{ReedSimon} \cite{Berezansky68} \cite{MassonRepka}.
\begin{definition}[Closable and closed operators]
    Let $X$ be a separable complex Hilbert space. Let $D\subset X$ be a linear subspace of $X$. An operator $T\in\mathcal{L}(D)$ is said to be \textit{closed} if its graph, given by
    \begin{align*}
        \operatorname{graph}(T):=\{(x,Tx) \mbox{ : } x\in D\}\subset X\times X
    \end{align*}
    is a closed subset of $X\times X$. 

    \noindent
    $T$ is said to be \textit{closable} if it admits a closed extension, i.e. if there exist a linear subspace $\tilde D$ with $D\subset \tilde D\subset X$ and a closed operator $\tilde T\in\mathcal{L}(\tilde D)$ such that $\tilde T=T$ on $D$.

    \noindent
    Every closable operator $T$ has a smallest closed extension, which we denote by $\overline{T}$. 
\end{definition}
\begin{definition}[Self-adjoint and skew-adjoint operators]\label{Definition: self and skew-adjoint operators}
    Let $X$ be a separable complex Hilbert space with inner product $(.,.)$. Let $D\subset X$ be a dense linear subspace of $X$. An operator $T\in\mathcal{L}(D)$ is called \textit{symmetric} if
    \begin{align*}
        ( Tx,y)=( x,Ty) \qquad\forall\,x,y\in D. 
    \end{align*}
    Analogously, $T$ is called \textit{skew-symmetric} if
    \begin{align*}
        ( Tx,y)=-( x,Ty) \qquad\forall\,x,y\in D. 
    \end{align*}

    \noindent
    A symmetric operator $T$ is called \textit{self-adjoint} if for every $y\in X\smallsetminus D$ there is no $z\in X$ such that 
    \begin{align*}
        ( Tx,y)=( x,z) \qquad\forall\,x\in D. 
    \end{align*}
    Analogously, a skew-symmetric operator $T$ is called \textit{skew-adjoint} if for every $y\in X\smallsetminus D$ there is no $z\in X$ such that 
    \begin{align*}
        ( Tx,y)=-( x,z) \qquad\forall\,x\in D. 
    \end{align*}
    
    \noindent
    A symmetric operator is called \textit{essentially self-adjoint} if its closure is self-adjoint. Analogously, a skew-symmetric operator $T$ is called \textit{essentially skew-adjoint} if its closure is skew-adjoint (the definition makes sense, since symmetric and skew-symmetric operators are always closable \cite[page 255]{ReedSimon}.) Note that $T$ is essentially skew-adjoint if and only if $iT$ is essentially self-adjoint.
\end{definition}

\noindent
For a separable complex Hilbert space $X$, we let
\begin{align*}
    \mathfrak{u}(X)&:=\{T \mbox{ is a skew-adjoint operator on } X\}\\
    \operatorname{U}(X)&:=\{T\in\mathcal{B}(X) \mbox{ : } T^*T=TT^*=\operatorname{Id}_X\}.
\end{align*}
\begin{remark}\label{Remark: exponential of skew-adjoint operators}
    By \cite[Theorem VIII.5]{ReedSimon}, for every $T\in\mathfrak{u}(H)$ the operator $\exp(T)$ is well-defined and belongs to $\operatorname{U}(H)$.
\end{remark}
\noindent
The following definition can be found in \cite[page 503]{Berezansky68}.
\begin{definition}[Jacobi matrix operator]\label{Definition: Jacobi matrix operator}
    We say that $J$ is a Jacobi matrix operator if it is a symmetric linear operator represented by an infinite tridiagonal matrix such that all diagonal terms $\{a_k\}_{k \in \mathbb{Z}}$ are real numbers and all off-diagonal terms $\{b_k\}_{k \in \mathbb{Z}}$ are positive real numbers.
\end{definition}
The following criterion for essential self-adjointness is well-known (see \cite[Corollary 2.2]{MassonRepka}). 
%
\begin{lemma}[Carleman's Test]\label{Lemma: Carleman's test}
    A Jacobi matrix operator $J$ is essentially self-adjoint if
    \begin{align*}
        \sum_{k\ge 0}\frac{1}{b_k}=\infty \qquad\mbox{ and } \qquad\sum_{k\le 0}\frac{1}{b_k}=\infty.
    \end{align*}
\end{lemma}
\begin{definition}[Reduced group $C^*$-algebra and reduced crossed product]
\label{Definition: reduced group algebra and crossed product}
    Let $\Gamma$ be a countable discrete group and let $\lambda_\Gamma$ be its regular representation on $\ell^2(\Gamma)$. The \textit{reduced group $C^*$-algebra} $C_r^*(\Gamma)$ is the closure of
    \begin{align*}
        \operatorname{span}_{\c}
        \{\lambda_\Gamma(\gamma)\mbox{ : }\gamma\in\Gamma\}
    \end{align*}
    in the operator norm of $\mathcal{B}(\ell^2(\Gamma))$.

    \noindent
    Suppose now that $\Gamma$ acts on $\s^1$ by homeomorphisms. Let $\mu=d\theta/(2\pi)$ be the normalized arc length measure on $\s^1$ and set $\mathcal{K}:=\ell^2(\Gamma;L^2(\s^1,\mu\,;\c))$. For every $\gamma\in\Gamma$ and every complex-valued continuous function $a\in C(\mathbb{S}^1)$, define the operators $u_\gamma,\pi(a)\in\mathcal{B}(\mathcal{K})$ by
    \begin{align*}
        (u_\gamma\Psi)(h,\xi)&:=\Psi(\gamma^{-1}h,\xi),\\
        (\pi(a)\Psi)(h,\xi)&:=a(h\xi)\Psi(h,\xi)
    \end{align*}
    for every $\Psi\in\mathcal{K}$, $h\in\Gamma$ and $\mu$-almost every $\xi\in\s^1$.
    The \textit{reduced crossed product} $C(\s^1)\rtimes_r\Gamma$ is the closure of
    \begin{align*}
        \operatorname{span}_{\c}\{\pi(a)u_\gamma\mbox{ : }a\in C(\s^1),\ \gamma\in\Gamma\}
    \end{align*}
    in the operator norm of $\mathcal{B}(\mathcal{K})$.
\end{definition}
\noindent
We identify $C(\s^1)$ with its image under $\pi$. Accordingly, we write $\sigma(a)$ for $\sigma(\pi(a))$. The defining formulas give
\begin{align*}
    u_\gamma\pi(a)u_\gamma^{-1}=\pi(a\circ\gamma^{-1}) \qquad\forall\,\gamma\in\Gamma,\quad\forall\,a\in C(\s^1).
\end{align*}
Moreover, under the natural identification
$\mathcal{K}\cong\ell^2(\Gamma)\otimes L^2(\s^1,\mu\,;\c)$, we have
\begin{align*}
    u_\gamma=\lambda_\Gamma(\gamma)\otimes\operatorname{Id}_{L^2(\s^1,\mu\,;\,\c)} \qquad\forall\,\gamma\in\Gamma.
\end{align*}
Hence, the map $b\mapsto b\otimes\operatorname{Id}_{L^2(\s^1,\mu\,;\,\c)}$ identifies $C_r^*(\Gamma)$ isometrically with the closed subalgebra of $C(\s^1)\rtimes_r\Gamma$ generated by the unitary operators $\{u_\gamma\}_{\gamma\in\Gamma}$. Since the regular representation is faithful on $\Gamma$, we identify each $\gamma\in\Gamma$ with $\lambda_\Gamma(\gamma)$ in $C_r^*(\Gamma)$ and hence with its image $u_\gamma$ in $C(\s^1)\rtimes_r\Gamma$. With this convention, for every $*$-representation $\sigma$ of $C(\s^1)\rtimes_r\Gamma$, we write $\sigma(\gamma)$ for $\sigma(u_\gamma)$.

\noindent
If $\rho$ is a unitary representation of $\Gamma$ on $H$ such that $\rho\prec\lambda_\Gamma$, then \cite[Theorem 7]{Harpe07} gives
\begin{align*}
    \bigg\|\sum_{\gamma\in\Gamma}c_\gamma\rho(\gamma)\bigg\|_{\mathcal{B}(H)}
    \leq
    \bigg\|\sum_{\gamma\in\Gamma}c_\gamma\lambda_\Gamma(\gamma)\bigg\|_{\mathcal{B}(\ell^2(\Gamma))}
\end{align*}
for every finitely supported family of complex coefficients $\{c_\gamma\}_{\gamma\in\Gamma}$. Thus, $\lambda_\Gamma(\gamma)\mapsto\rho(\gamma)$ extends by linearity and continuity to a unital $*$-representation of $C_r^*(\Gamma)$ on $H$, namely a linear map into $\mathcal{B}(H)$ preserving products, adjoints and the identity.
\section{The isometry groups of simply connected, 2-dimensional space forms}
For the reader's convenience, we collect in this section useful information about the isometry groups of the round 2-sphere and the hyperbolic 2-plane. Although we do not need the case of the sphere, it may be helpful to compare its properties with those of the hyperbolic plane. In particular, we recall the following:
\begin{enumerate}
    \item $\operatorname{SO}(3)$ is the orientation preserving isometry group of the round 2-sphere $\mathbb{S}^2$, 
    \item $\PSL(2,\mathbb{R})$ is the orientation preserving isometry group of the hyperbolic plane $\mathbb{H}^2$.
\end{enumerate}
Their Lie algebras are respectively given by 
\begin{enumerate}
    \item $\mathfrak{so}(3)=\{A\in\operatorname{Mat}_3(\mathbb{R}) \mbox{ : } A^{T}=-A\}$, 
    \item $\mathfrak{sl}(2,\mathbb{R})=\{A\in\operatorname{Mat}_2(\mathbb{R}) \mbox{ : } \operatorname{tr}(A)=0\}$.
\end{enumerate}
Next, we record the following classification results for irreducible representations of the Lie groups above on finite dimensional and infinite dimensional Hilbert spaces.
\begin{proposition}[The irreducible unitary representations of $\operatorname{SO}(3)$]
    Let $H$ be a real separable Hilbert space. Then, the following facts hold.
    \begin{enumerate}
        \item If $H$ is finite dimensional and $\dim H$ is even or if $H$ is infinite dimensional, there is no irreducible unitary representation of $\operatorname{SO}(3)$ on $H$.
        \item If $H$ is finite dimensional and $\dim H=n$ is odd, there is a unique (up to strong equivalence) unitary irreducible representation $\rho_n$ of $\operatorname{SO}(3)$ on $H$. If $H_{\ell}$ denotes the Hilbert space of the spherical harmonics on the unit round sphere $\,\mathbb{S}^2$ of degree $\ell=\frac{n-1}{2}$, we have $H\cong H_{\ell}$, and the representation $\rho_n$ is given by 
        \begin{align*}
            \rho_n(A):=\big\{H_{\ell}\ni\psi\mapsto\psi(A^{-1}\,\cdot\,)\in H_{\ell}\big\}\in\operatorname{GL}(H_{\ell}) \qquad\forall\,A\in\operatorname{SO}(3).
        \end{align*}
    \end{enumerate}
\end{proposition}
\begin{proof}
    The statement is a direct consequence of \cite[Theorem 4.32]{Hall} and \cite[Proposition 4.35]{Hall}.
\end{proof}
\begin{theorem}[The irreducible unitary representations of $\operatorname{PSL}(2,\mathbb{R})$]\label{thm: irreducible representations} 
    Let $H$ be a separable Hilbert space. Then, the following facts hold.
    \begin{enumerate}
        \item If $H$ is finite dimensional, there is no nontrivial irreducible unitary representation of $\operatorname{PSL}(2,\mathbb{R})$ on $H$.
        \item If $H$ is infinite dimensional, the irreducible unitary representations of $\operatorname{PSL}(2,\mathbb{R})$ on $H$ having a $\PSO(2)$-invariant unit vector---which we call \textit{of class one}---are given by the following parametrizations.
        \begin{enumerate}[(a)]
            \item The principal series $\{\rho_s\}_{s\in i\mathbb{R}}$, where each $\rho_s$ corresponds to the unique radial eigenfunction $\varphi_{s}$ of the Laplace--Beltrami operator on the Poincar\'e disk $\mathbb{D}$ with eigenvalue 
            \begin{align}
                \lambda_s:=\frac{1}{4}-s^2\ge\frac{1}{4}
            \end{align}
            such that $\varphi_s(0)=1$.
            \item The \textit{complementary series} $\{\rho_s\}_{s\in (-\frac{1}{2},\frac{1}{2})}$, where each $\rho_s$ corresponds to the unique radial eigenfunction $\varphi_{s}$ of the Laplace--Beltrami operator on the Poincar\'e disk $\mathbb{D}$ with eigenvalue 
            \begin{align}
                0<\lambda_s:=\frac{1}{4}-s^2\le\frac{1}{4}
            \end{align}
            such that $\varphi_s(0)=1$.
        \end{enumerate}
    \end{enumerate}
\end{theorem}
\begin{proof}
    See \cite[Theorem 8, page 123]{Lang} or \cite[page 36]{Knapp} or \cite[Chapter 5]{Lubotzky}.
\end{proof}
Lastly, for further use in the proof of our main theorem, we give explicit sets of generators of the Lie algebras $\mathfrak{so}(3)$ and $\mathfrak{sl}(2,\mathbb{R})$ exhibiting structure constants which are  convenient to our purposes. 
\begin{lemma}\label{Lemma: a useful basis of so(3)}
    Let $\tau_1,\tau_2,\tau_3\in \mathfrak{so}(3)$ be the generators of the Lie algebra $\mathfrak{so}(3)$ given by
    \begin{align*}
        \tau_1:=\begin{pmatrix}
            0&0&0\\
            0&0&1\\
            0&-1&0
         \end{pmatrix},\qquad 
        \tau_2:=\begin{pmatrix}
            0&0&-1\\
            0&0&0\\
            1&0&0
        \end{pmatrix},\qquad
        \tau_3:=\begin{pmatrix}
            0&1&0\\
            -1&0&0\\
            0&0&0
            \end{pmatrix}. 
    \end{align*}
    Then, the following commuting relations hold:
    \begin{align*}
        [\tau_1,\tau_2]=-\tau_3 \qquad [\tau_2,\tau_3]=-\tau_1 \qquad [\tau_3,\tau_1]=-\tau_2.
    \end{align*}
\end{lemma}
\begin{lemma}\label{Lemma: a useful basis of sl(2,R)}
    Let $\sigma_1,\sigma_2,\sigma_3\in\mathfrak{sl}(2,\mathbb{R})$ be the generators of the Lie algebra $\mathfrak{sl}(2,\mathbb{R})$ given by
    \begin{align*}
        \sigma_1:=\begin{pmatrix}
            0&\frac{1}{2}\\
            \frac{1}{2}&0
         \end{pmatrix},\qquad 
        \sigma_2:=\begin{pmatrix}
            -\frac{1}{2}&0\\
            0&\frac{1}{2} 
        \end{pmatrix},\qquad
        \sigma_3:=\begin{pmatrix}
            0&\frac{1}{2}\\
            -\frac{1}{2}&0
            \end{pmatrix}. 
    \end{align*}
    Then, the following commuting relations hold:
    \begin{align*}
        [\sigma_1,\sigma_2]=\sigma_3 \qquad [\sigma_2,\sigma_3]=-\sigma_1 \qquad [\sigma_3,\sigma_1]=-\sigma_2.
    \end{align*}
\end{lemma}
\bibliographystyle{amsalpha} 
\bibliography{main} 

@book{CM11,
	author = {Colding, Tobias H and Minicozzi, William P},
	publisher = {American Mathematical Soc.},
	title = {A course in minimal surfaces},
	volume = {121},
	year = {2011}}

@article {Calabi,
    AUTHOR = {Calabi, Eugenio},
     TITLE = {Minimal immersions of surfaces in {E}uclidean spheres},
   JOURNAL = {J. Differential Geometry},
  FJOURNAL = {Journal of Differential Geometry},
    VOLUME = {1},
      YEAR = {1967},
     PAGES = {111--125},
      ISSN = {0022-040X},
   MRCLASS = {53.04},
  MRNUMBER = {233294},
MRREVIEWER = {A. Gray},
       URL = {http://projecteuclid.org/euclid.jdg/1214427884},
}

@article {Bryant,
    AUTHOR = {Bryant, Robert L.},
     TITLE = {Minimal surfaces of constant curvature in {$S^n$}},
   JOURNAL = {Trans. Amer. Math. Soc.},
  FJOURNAL = {Transactions of the American Mathematical Society},
    VOLUME = {290},
      YEAR = {1985},
    NUMBER = {1},
     PAGES = {259--271},
      ISSN = {0002-9947},
   MRCLASS = {53C42 (53A10)},
  MRNUMBER = {787964},
MRREVIEWER = {Katsuei Kenmotsu},
       DOI = {10.2307/1999793},
       URL = {https://doi.org/10.2307/1999793},
}

@book {Lubotzky,
    AUTHOR = {Lubotzky, Alexander},
     TITLE = {Discrete groups, expanding graphs and invariant measures},
    SERIES = {Modern Birkh\"{a}user Classics},
      NOTE = {With an appendix by Jonathan D. Rogawski,
              Reprint of the 1994 edition},
 PUBLISHER = {Birkh\"{a}user Verlag, Basel},
      YEAR = {2010},
     PAGES = {iii+192},
      ISBN = {978-3-0346-0331-7},
   MRCLASS = {22E40 (05C25 11F70 28C10 43A07)},
  MRNUMBER = {2569682},
       DOI = {10.1007/978-3-0346-0332-4},
       URL = {https://doi.org/10.1007/978-3-0346-0332-4},
}

@book {Lang,
    AUTHOR = {Lang, Serge},
     TITLE = {{${\rm SL}_2({\bf R})$}},
    SERIES = {Graduate Texts in Mathematics},
    VOLUME = {105},
      NOTE = {Reprint of the 1975 edition},
 PUBLISHER = {Springer-Verlag, New York},
      YEAR = {1985},
     PAGES = {xiv+428},
      ISBN = {0-387-96198-4},
   MRCLASS = {22E45 (11F72)},
  MRNUMBER = {803508},
}

@misc{Song25,
      title={Random harmonic maps into spheres}, 
      author={Antoine Song},
      year={2025},
      eprint={2402.10287},
      archivePrefix={arXiv},
      primaryClass={math.DG},
      url={https://arxiv.org/abs/2402.10287}, 
}

@misc{Song24, 
      title={Hyperbolic groups and spherical minimal surfaces}, 
      author={Antoine Song},
      year={2024},
      eprint={2402.10869},
      archivePrefix={arXiv},
      primaryClass={math.DG},
      url={https://arxiv.org/abs/2402.10869}, 
}

@book{BdLHV08, 
    place={Cambridge}, 
    series={New Mathematical Monographs}, 
    title={Kazhdan’s Property {($T$)}}, 
    publisher={Cambridge University Press}, author={Bekka, Bachir and de la Harpe, Pierre and Valette, Alain}, year={2008}, collection={New Mathematical Monographs},
}

@article{BCG95,
  title={Entropies et rigidit{\'e}s des espaces localement sym{\'e}triques de courbure strictement n{\'e}gative},
  author={G{\'e}rard Besson and Gilles Courtois and Sylvestre Gallot},
  journal={Geometric \& Functional Analysis GAFA},
  year={1995},
  volume={5},
  pages={731-799},
  url={https://api.semanticscholar.org/CorpusID:122161246},
}

@misc{Son23b, 
    title={Spherical volume and spherical {P}lateau problem},
    author={Song, Antoine},
    note={to appear in S\'{e}minaire de th\'{e}orie spectrale et g\'{e}om\'{e}trie}, 
    year={2023},
}

@article {Kenmotsu,
    AUTHOR = {Kenmotsu, Katsuei},
     TITLE = {On minimal immersions of {$R^{2}$} into {$S^{N}$}},
   JOURNAL = {J. Math. Soc. Japan},
  FJOURNAL = {Journal of the Mathematical Society of Japan},
    VOLUME = {28},
      YEAR = {1976},
    NUMBER = {1},
     PAGES = {182--191},
      ISSN = {0025-5645},
   MRCLASS = {49F10 (53C40)},
  MRNUMBER = {405218},
MRREVIEWER = {Jo\~{a}o Lucas Marques Barbosa},
       DOI = {10.2969/jmsj/02810182},
       URL = {https://doi.org/10.2969/jmsj/02810182},
}

@article {Adams94,
    AUTHOR = {Adams, S.},
     TITLE = {Boundary amenability for word hyperbolic groups and an
              application to smooth dynamics of simple groups},
   JOURNAL = {Topology},
  FJOURNAL = {Topology. An International Journal of Mathematics},
    VOLUME = {33},
      YEAR = {1994},
    NUMBER = {4},
     PAGES = {765--783},
      ISSN = {0040-9383},
   MRCLASS = {58F11 (20F32 22E99 57M50)},
  MRNUMBER = {1293309},
MRREVIEWER = {Ralf\ J.\ Spatzier},
       DOI = {10.1016/0040-9383(94)90007-8},
       URL = {https://doi.org/10.1016/0040-9383(94)90007-8},
}

@article {Kuhn94,
    AUTHOR = {Kuhn, M. Gabriella},
     TITLE = {Amenable actions and weak containment of certain
              representations of discrete groups},
   JOURNAL = {Proc. Amer. Math. Soc.},
  FJOURNAL = {Proceedings of the American Mathematical Society},
    VOLUME = {122},
      YEAR = {1994},
    NUMBER = {3},
     PAGES = {751--757},
      ISSN = {0002-9939,1088-6826},
   MRCLASS = {43A07 (22D10)},
  MRNUMBER = {1209424},
MRREVIEWER = {A.\ Hulanicki},
       DOI = {10.2307/2160750},
       URL = {https://doi.org/10.2307/2160750},
}

@misc{Lukasz16,
      title={Boundary representations of hyperbolic groups}, 
      author={Lukasz Garncarek},
      year={2016},
      eprint={1404.0903},
      archivePrefix={arXiv},
      primaryClass={math.DS},
      url={https://arxiv.org/abs/1404.0903}, 
}

@article {Harpe07,
    AUTHOR = {de la Harpe, Pierre},
     TITLE = {On simplicity of reduced {$C^*$}-algebras of groups},
   JOURNAL = {Bull. Lond. Math. Soc.},
  FJOURNAL = {Bulletin of the London Mathematical Society},
    VOLUME = {39},
      YEAR = {2007},
    NUMBER = {1},
     PAGES = {1--26},
      ISSN = {0024-6093,1469-2120},
   MRCLASS = {22D25 (46L05)},
  MRNUMBER = {2303514},
MRREVIEWER = {Bachir Bekka},
       DOI = {10.1112/blms/bdl014},
       URL = {https://doi.org/10.1112/blms/bdl014},
}

@incollection {Harpe85,
    AUTHOR = {de la Harpe, Pierre},
     TITLE = {Reduced {$C^*$}-algebras of discrete groups which are simple with a unique trace},
 BOOKTITLE = {Operator algebras and their connections with topology and
              ergodic theory (Busteni, 1983)},
    SERIES = {Lecture Notes in Math.},
    VOLUME = {1132},
     PAGES = {230--253},
 PUBLISHER = {Springer, Berlin},
      YEAR = {1985},
      ISBN = {3-540-15643-7},
   MRCLASS = {22D25 (20F38 22E40 46L05)},
  MRNUMBER = {799571},
MRREVIEWER = {Jonathan\ M.\ Rosenberg},
       DOI = {10.1007/BFb0074887},
       URL = {https://doi.org/10.1007/BFb0074887},
}

@article {Bargmann47,
    AUTHOR = {Bargmann, V.},
     TITLE = {Irreducible unitary representations of the {L}orentz group},
   JOURNAL = {Ann. of Math. (2)},
  FJOURNAL = {Annals of Mathematics. Second Series},
    VOLUME = {48},
      YEAR = {1947},
     PAGES = {568--640},
      ISSN = {0003-486X},
   MRCLASS = {20.0X},
  MRNUMBER = {21942},
MRREVIEWER = {H.\ Samelson},
       DOI = {10.2307/1969129},
       URL = {https://doi.org/10.2307/1969129},
}

@article {Chandra52,
    AUTHOR = {Harish-Chandra},
     TITLE = {Plancherel formula for the {$2\times 2$} real unimodular
              group},
   JOURNAL = {Proc. Nat. Acad. Sci. U.S.A.},
  FJOURNAL = {Proceedings of the National Academy of Sciences of the United
              States of America},
    VOLUME = {38},
      YEAR = {1952},
     PAGES = {337--342},
      ISSN = {0027-8424},
   MRCLASS = {20.0X},
  MRNUMBER = {47055},
MRREVIEWER = {G.\ W.\ Mackey},
       DOI = {10.1073/pnas.38.4.337},
       URL = {https://doi.org/10.1073/pnas.38.4.337},
}

@article {GN46,
    AUTHOR = {Gelfand, I. and Naimark, M.},
     TITLE = {Unitary representations of the {L}orentz group},
   JOURNAL = {Acad. Sci. USSR. J. Phys.},
  FJOURNAL = {Acad. Sci. USSR. J. Phys.},
    VOLUME = {10},
      YEAR = {1946},
     PAGES = {93--94},
   MRCLASS = {20.0X},
  MRNUMBER = {17282},
MRREVIEWER = {E.\ P.\ Wigner},
}

@incollection {KS92,
    AUTHOR = {Kuhn, Gabriella and Steger, Tim},
     TITLE = {Boundary representations of the free group. {I}, {II}},
 BOOKTITLE = {Harmonic analysis and discrete potential theory ({F}rascati, 1991)},
     PAGES = {85--91, 93--97},
 PUBLISHER = {Plenum, New York},
      YEAR = {1992},
      ISBN = {0-306-44225-6},
   MRCLASS = {22D10 (20E05 22D25)},
  MRNUMBER = {1222451},
MRREVIEWER = {Nobuaki\ Obata},
}

@article {PS86,
    AUTHOR = {Pytlik, T. and Szwarc, R.},
     TITLE = {An analytic family of uniformly bounded representations of
              free groups},
   JOURNAL = {Acta Math.},
  FJOURNAL = {Acta Mathematica},
    VOLUME = {157},
      YEAR = {1986},
    NUMBER = {3-4},
     PAGES = {287--309},
      ISSN = {0001-5962,1871-2509},
   MRCLASS = {22E35 (22D12 22E50 43A65 43A90)},
  MRNUMBER = {857676},
MRREVIEWER = {A.\ Fig\`a-Talamanca},
       DOI = {10.1007/BF02392596},
       URL = {https://doi.org/10.1007/BF02392596},
}

@article {Szwarc88,
    AUTHOR = {Szwarc, Ryszard},
     TITLE = {An analytic series of irreducible representations of the free
              group},
   JOURNAL = {Ann. Inst. Fourier (Grenoble)},
  FJOURNAL = {Universit\'e{} de Grenoble. Annales de l'Institut Fourier},
    VOLUME = {38},
      YEAR = {1988},
    NUMBER = {1},
     PAGES = {87--110},
      ISSN = {0373-0956,1777-5310},
   MRCLASS = {22E40 (22D12 43A90)},
  MRNUMBER = {949000},
MRREVIEWER = {A.\ Fig\`a-Talamanca},
       DOI = {10.5802/aif.1124},
       URL = {https://doi.org/10.5802/aif.1124},
}

@article {Siu80,
    AUTHOR = {Siu, Yum Tong},
     TITLE = {The complex-analyticity of harmonic maps and the strong
              rigidity of compact {K}\"ahler manifolds},
   JOURNAL = {Ann. of Math. (2)},
  FJOURNAL = {Annals of Mathematics. Second Series},
    VOLUME = {112},
      YEAR = {1980},
    NUMBER = {1},
     PAGES = {73--111},
      ISSN = {0003-486X},
   MRCLASS = {53C55 (32H99 58E20)},
  MRNUMBER = {584075},
MRREVIEWER = {M.\ Kalka},
       DOI = {10.2307/1971321},
       URL = {https://doi.org/10.2307/1971321},
}

@article {Corlette92,
    AUTHOR = {Corlette, Kevin},
     TITLE = {Archimedean superrigidity and hyperbolic geometry},
   JOURNAL = {Ann. of Math. (2)},
  FJOURNAL = {Annals of Mathematics. Second Series},
    VOLUME = {135},
      YEAR = {1992},
    NUMBER = {1},
     PAGES = {165--182},
      ISSN = {0003-486X,1939-8980},
   MRCLASS = {57S30 (22E40 53C25 57M50 58E20)},
  MRNUMBER = {1147961},
MRREVIEWER = {Christopher\ W.\ Stark},
       DOI = {10.2307/2946567},
       URL = {https://doi.org/10.2307/2946567},
}

@article {MSY93,
    AUTHOR = {Mok, Ngaiming and Siu, Yum Tong and Yeung, Sai-Kee},
     TITLE = {Geometric superrigidity},
   JOURNAL = {Invent. Math.},
  FJOURNAL = {Inventiones Mathematicae},
    VOLUME = {113},
      YEAR = {1993},
    NUMBER = {1},
     PAGES = {57--83},
      ISSN = {0020-9910,1432-1297},
   MRCLASS = {53C42 (53C35 58E20)},
  MRNUMBER = {1223224},
MRREVIEWER = {Yuan-Jen\ Chiang},
       DOI = {10.1007/BF01244302},
       URL = {https://doi.org/10.1007/BF01244302},
}

@article {GS92,
    AUTHOR = {Gromov, Mikhail and Schoen, Richard},
     TITLE = {Harmonic maps into singular spaces and {$p$}-adic
              superrigidity for lattices in groups of rank one},
   JOURNAL = {Inst. Hautes \'Etudes Sci. Publ. Math.},
  FJOURNAL = {Institut des Hautes \'Etudes Scientifiques. Publications
              Math\'ematiques},
    NUMBER = {76},
      YEAR = {1992},
     PAGES = {165--246},
      ISSN = {0073-8301,1618-1913},
   MRCLASS = {58E20 (22E40)},
  MRNUMBER = {1215595},
MRREVIEWER = {Caio\ J. C. Negreiros},
       URL = {http://www.numdam.org/item?id=PMIHES_1992__76__165_0},
}

@article {KS93,
    AUTHOR = {Korevaar, Nicholas J. and Schoen, Richard M.},
     TITLE = {Sobolev spaces and harmonic maps for metric space targets},
   JOURNAL = {Comm. Anal. Geom.},
  FJOURNAL = {Communications in Analysis and Geometry},
    VOLUME = {1},
      YEAR = {1993},
    NUMBER = {3-4},
     PAGES = {561--659},
      ISSN = {1019-8385,1944-9992},
   MRCLASS = {58E20 (46E35)},
  MRNUMBER = {1266480},
MRREVIEWER = {Guojun\ Liao},
       DOI = {10.4310/CAG.1993.v1.n4.a4},
       URL = {https://doi.org/10.4310/CAG.1993.v1.n4.a4},
}

@article {KS97,
    AUTHOR = {Korevaar, Nicholas J. and Schoen, Richard M.},
     TITLE = {Global existence theorems for harmonic maps to non-locally
              compact spaces},
   JOURNAL = {Comm. Anal. Geom.},
  FJOURNAL = {Communications in Analysis and Geometry},
    VOLUME = {5},
      YEAR = {1997},
    NUMBER = {2},
     PAGES = {333--387},
      ISSN = {1019-8385,1944-9992},
   MRCLASS = {58E20},
  MRNUMBER = {1483983},
MRREVIEWER = {Raul\ Quiroga-Barranco},
       DOI = {10.4310/CAG.1997.v5.n2.a4},
       URL = {https://doi.org/10.4310/CAG.1997.v5.n2.a4},
}

@article {Wang98,
    AUTHOR = {Wang, Mu-Tao},
     TITLE = {A fixed point theorem of discrete group actions on
              {R}iemannian manifolds},
   JOURNAL = {J. Differential Geom.},
  FJOURNAL = {Journal of Differential Geometry},
    VOLUME = {50},
      YEAR = {1998},
    NUMBER = {2},
     PAGES = {249--267},
      ISSN = {0022-040X,1945-743X},
   MRCLASS = {53C24 (53C21 58E20)},
  MRNUMBER = {1684980},
MRREVIEWER = {Raul\ Quiroga-Barranco},
       URL = {http://projecteuclid.org/euclid.jdg/1214461170},
}

@incollection {JY93,
    AUTHOR = {Jost, J\"urgen and Yau, Shing-Tung},
     TITLE = {Harmonic maps and superrigidity},
 BOOKTITLE = {Differential geometry: partial differential equations on
              manifolds ({L}os {A}ngeles, {CA}, 1990)},
    SERIES = {Proc. Sympos. Pure Math.},
    VOLUME = {54, Part 1},
     PAGES = {245--280},
 PUBLISHER = {Amer. Math. Soc., Providence, RI},
      YEAR = {1993},
      ISBN = {0-8218-1494-X},
   MRCLASS = {58E20 (53C35)},
  MRNUMBER = {1216587},
MRREVIEWER = {Yuan-Jen\ Chiang},
       DOI = {10.1090/pspum/054.1/1216587},
       URL = {https://doi.org/10.1090/pspum/054.1/1216587},
}

@article {DMV11,
    AUTHOR = {Daskalopoulos, Georgios and Mese, Chikako and Vdovina, Alina},
     TITLE = {Superrigidity of hyperbolic buildings},
   JOURNAL = {Geom. Funct. Anal.},
  FJOURNAL = {Geometric and Functional Analysis},
    VOLUME = {21},
      YEAR = {2011},
    NUMBER = {4},
     PAGES = {905--919},
      ISSN = {1016-443X,1420-8970},
   MRCLASS = {53C24 (58E20)},
  MRNUMBER = {2827014},
MRREVIEWER = {Anders\ Karlsson},
       DOI = {10.1007/s00039-011-0124-9},
       URL = {https://doi.org/10.1007/s00039-011-0124-9},
}

@misc{BDM24,
      title={Harmonic Maps into Euclidean Buildings and Non-Archimedean Superrigidity}, 
      author={Christine Breiner and Ben K. Dees and Chikako Mese},
      year={2024},
      eprint={2408.02783},
      archivePrefix={arXiv},
      primaryClass={math.DG},
      url={https://arxiv.org/abs/2408.02783}, 
}

@book {Knapp,
    AUTHOR = {Knapp, Anthony W.},
     TITLE = {Representation theory of semisimple groups},
    SERIES = {Princeton Mathematical Series},
    VOLUME = {36},
      NOTE = {An overview based on examples},
 PUBLISHER = {Princeton University Press, Princeton, NJ},
      YEAR = {1986},
     PAGES = {xviii+774},
      ISBN = {0-691-08401-7},
   MRCLASS = {22E46 (22-01 22E30)},
  MRNUMBER = {855239},
MRREVIEWER = {Wulf Rossmann},
       DOI = {10.1515/9781400883974},
       URL = {https://doi.org/10.1515/9781400883974},
}

@book{ReedSimon,
   title =     {Functional {A}nalysis},
   author =    {Reed, M. and Simon, B.},
   publisher = {Academic Press},
   isbn =      {9780125850506; 0125850506},
   year =      {1980},
   series =    {Methods of Modern Mathematical Physics, Volume 1},
}

@article{ArendtKreuter,
   title={Mapping theorems for {S}obolev spaces of vector-valued functions},
   volume={240},
   ISSN={1730-6337},
   url={http://dx.doi.org/10.4064/sm8757-4-2017},
   DOI={10.4064/sm8757-4-2017},
   number={3},
   journal={Studia Mathematica},
   publisher={Institute of Mathematics, Polish Academy of Sciences},
   author={Arendt, Wolfgang and Kreuter, Marcel},
   year={2018},
   pages={275–299} 
}

@book {Berezansky68,
    AUTHOR = {Berezansky, Y.M.},
     TITLE = {Expansions in eigenfunctions of selfadjoint operators},
    SERIES = {Translations of Mathematical Monographs},
    VOLUME = {Vol. 17},
      NOTE = {Translated from the Russian by R. Bolstein, J. M. Danskin, J.
              Rovnyak and L. Shulman},
 PUBLISHER = {American Mathematical Society, Providence, RI},
      YEAR = {1968},
     PAGES = {ix+809},
   MRCLASS = {47.65 (34.00)},
  MRNUMBER = {222718},
}

@article{MassonRepka,
author = {Masson, D. R. and Repka, J.},
TITLE = {Spectral {T}heory of {J}acobi {M}atrices in $l^2(\mathbb{Z})$ and the $su(1,1)$ {L}ie {A}lgebra},
journal = {SIAM Journal on Mathematical Analysis},
volume = {22},
number = {4},
pages = {1131-1146},
year = {1991},
doi = {10.1137/0522073},

URL = {https://doi.org/10.1137/0522073},
}

@book {Hall,
    AUTHOR = {Hall, Brian},
     TITLE = {Lie groups, {L}ie algebras, and representations},
    SERIES = {Graduate Texts in Mathematics},
    VOLUME = {222},
   EDITION = {Second},
      NOTE = {An elementary introduction},
 PUBLISHER = {Springer, Cham},
      YEAR = {2015},
     PAGES = {xiv+449},
      ISBN = {978-3-319-13466-6; 978-3-319-13467-3},
   MRCLASS = {22-01 (17-01)},
  MRNUMBER = {3331229},
       DOI = {10.1007/978-3-319-13467-3},
       URL = {https://doi.org/10.1007/978-3-319-13467-3},
}

@article {BS71,
    AUTHOR = {Bochnak, Jacek and Siciak, J\'ozef},
     TITLE = {Analytic functions in topological vector spaces},
   JOURNAL = {Studia Math.},
  FJOURNAL = {Polska Akademia Nauk. Instytut Matematyczny. Studia
              Mathematica},
    VOLUME = {39},
      YEAR = {1971},
     PAGES = {77--112},
      ISSN = {0039-3223,1730-6337},
   MRCLASS = {46G05 (32A99 47H99)},
  MRNUMBER = {313811},
MRREVIEWER = {S.\ Dineen},
       DOI = {10.4064/sm-39-1-77-112},
       URL = {https://doi.org/10.4064/sm-39-1-77-112},
}

@article{Arveson69,
    author = {Arveson, William B.},
    title = {Subalgebras of {$C^*$}-algebras},
    journal = {Acta Math.},
    volume = {123},
    year = {1969},
    pages = {141--224},
    doi = {10.1007/BF02392388}
}
\end{document}